\documentclass[11pt]{article}
\usepackage{natbib} % bibtex
\usepackage{amssymb,amsmath,amsthm} % mathbb et.al.
\usepackage{geometry} % change paper size
\usepackage{color}
\usepackage{graphicx}
\usepackage{tikz}
\usepackage{bm}
\usepackage{rotating}
\usepackage{multirow}

\newtheorem{theorem}{Theorem}
\newtheorem{proposition}[theorem]{Proposition}
\newtheorem{definition}[theorem]{Definition}

\newtheorem{example}[theorem]{Example}

\title{Total loss estimation using copula-based regression models\thanks{The second author gratefully acknowledges the support of the TUM Graduate School's International School of Applied Mathematics as well as of Allianz Deutschland AG. The numerical computations were performed on a Linux cluster supported by DFG grant INST 95/919-1 FUGG.}}
\author{Nicole Kr\"amer $\quad$ Eike C. Brechmann $\quad$ Daniel Silvestrini $\quad$ Claudia Czado\\
Department of Mathematical Statistics\\
Technische Universit\"at M\"unchen}
\begin{document}
\maketitle
\begin{abstract}
We present a joint copula-based model for insurance claims and sizes. It uses bivariate copulae to accommodate for the dependence between these quantities. We derive the general distribution of the policy loss without the restrictive assumption of independence. We illustrate that this distribution tends to be skewed and multi-modal, and that an independence assumption can lead to substantial bias in the estimation of the policy loss. Further, we extend our framework to regression models by combining marginal generalized linear models with a copula. We show that this approach leads to a flexible class of models, and that the parameters can be estimated efficiently using maximum-likelihood. We propose a test procedure for the selection of the optimal copula family. The usefulness of our approach is illustrated in a simulation study and in an analysis of car insurance policies.
\end{abstract}

{\bf Keywords}: dependence modeling; generalized linear model; number of claims; claim size; policy loss

% \linenumbers

%% main text
\section{Introduction}
Estimating the total loss of an insurance portfolio is crucial for many actuarial decisions, e.g. for pricing of insurance contracts and for the calculation of premiums. A very common approach, based on the compound model by \citet{lun03}, models the average claim size and the number of claims independently, and then defines the loss as the product of these two quantities. However, the assumption of independence can be too restrictive and  lead to a systematic  over- or under-estimation of the policy loss. Evidently, this effects the accuracy of the  estimation of the portfolio loss.

We therefore propose a joint model that explicitly allows a dependency between average claim sizes and number of claims. This is achieved by combining marginal distributions for claim frequency and severity with families of bivariate copulae. A main contribution of this paper is the derivation of the distribution of the loss of an insurance policy. We illustrate that the distribution is often very skewed, and that -- depending on the model parameters -- the distribution is multi-modal. Based on this distribution, we can estimate the expected policy loss and its quantiles. We show that the distribution, and in particular its mean, depends strongly on the degree of dependence. This underpins the usefulness of our copula-based model.

Dependence modeling using copulae has become very popular the last years (see the standard reference books by \cite{har97} and \cite{Nelsen2006}) and was introduced to actuarial mathematics by \cite{FreesValdez98}.
Since then, copulae have been used, e.g., for the modeling of bivariate loss distributions by \cite{Klugman1999} and of dependencies between loss triangles by \cite{DeJong2009}, as well as for risk management (see \cite{McNeilFreyEmbrechts2005}).

It is common practice to model average claim sizes and number of claims in terms of a set of covariates as e.g. gender or age (see, e.g., \cite{Haberman1996} for an overview). Typically, claim frequency and severity are however modeled separately under the independence assumption of \cite{lun03}. \cite{gschl07} therefore included the number of claims as a covariate
into the model for average claim size. To allow for more flexibility and generality in the type of dependence, we extend our copula-based model to regression models by combining generalized linear models for the two marginal regression models with copula families. This is an extension of a recent approach by \cite{czakas10} and \cite{DeLeonWu2011} who only consider a Gauss copula based on work by \cite{Song00,Song2007} and \cite{Song2009}.

In our general copula-based regression approach, the model parameters  can be estimated efficiently using maximum-likelihood techniques. Further, we provide asymptotic confidence intervals that allow us to quantify the uncertainty of our estimates. For the selection of the copula family, we propose the likelihood-ratio test by \cite{vuong89}.

In an extensive simulation study, we show that the incorporation of a copula allows a more precise estimation of the individual policy losses, which in turn leads to a more reliable estimation of the total loss. These results are confirmed in a case study on car insurance policies.

%--------------------------------------------------------%
\section{Bivariate copulae for continuous-discrete data} %
%--------------------------------------------------------%
\subsection{Background: bivariate copulae}
 A bivariate copula $C: [0,1] \times [0,1]\rightarrow[0,1]$  is a bivariate cumulative distribution function on $[0,1] \times [0,1] $ with uniformly distributed margins. The importance of copulae is underpinned by Sklar's Theorem \citeyearpar{skl57}. In the bivariate case, it states that for every joint distribution function $F_{X,Y}$ of a bivariate random variable $(X,Y)$ with  univariate marginal distribution functions $F_{X}$ and $F_Y$, there exists a bivariate copula $C$ such that
\begin{eqnarray}
F_{X,Y}(x,y)&=&C(F_X(x),F_Y(y)).
\label{eq:Copula}
\end{eqnarray}
If $X$ and $Y$ are continuous random variables, the copula $C$ is unique. Conversely, if $C$ is a copula, Equation \eqref{eq:Copula} defines a bivariate distribution with marginal distribution functions $F_X$ and $F_Y$.  This allows us to model the marginal distributions and the joint dependence separately, as we can define the copula $C$ independently of the marginal distributions.

Copulae are invariant under monotone transformations of the marginal distributions. Therefore, instead of the  correlation coefficient -- which measures linear associations -- monotone association measures are used. A very common choice is Kendall's $\tau$,
\begin{eqnarray*}
\tau&:=&4 \int_{[0,1]^2} C(u,v)dC(u,v) -1 \in [-1,1] \,.
\end{eqnarray*}
In this paper, we study copula-based models for a pair of continuous-discrete random variables. We denote the continuous random variable by $X$, and the discrete random variable by $Y$. We assume that $Y$ takes values in $1,2,\ldots$. Their joint distribution is defined by a parametric copula  $C(\cdot,\cdot|\theta)$ that depends on a parameter $\theta$, i.e. the joint distribution is given by
\begin{eqnarray*}
F_{X,Y|\theta}(x,y)&=& C\left(F_X(x),F_Y(y) |\theta\right).
\end{eqnarray*}
We focus on four families of parametric bivariate copulae, namely the Clayton, Gumbel, Frank and Gauss copulae. Each family depends on a copula parameter $\theta$.  These parameters can be expressed in terms of Kendall's $\tau$. The definitions of the copula families and their relationship to Kendall's $\tau$ are provided in   \ref{app:copulae}. We note that the Clayton copula is only defined for positive values of Kendall's $\tau$, and that the Gumbel copula is only defined for non-negative values of Kendall's $\tau$. Via a rotation, it is however possible to extend these copula families to negative values of $\tau$. An overview on bivariate copula families and their properties, in particular their different tail behavior, can be found e.g. in \cite{BreSch12}.

For sampling, estimation and prediction, we need  the joint density/probability mass function of $X$ and $Y$ that is defined as
\begin{eqnarray}
\label{eq:density_mass}
f_{X,Y}(x,y)&:=&\frac{\partial}{\partial x}P(X\leq x,Y=y)\,.
\end{eqnarray}
In the remainder of this paper, we will refer to $f$ as the joint density function of $X$ and $Y$.

We now derive formulas for the joint density of $X$ and $Y$ in terms of the copula $C(\cdot,\cdot|\theta)$. We denote by
\begin{eqnarray}
\label{eq:parder}
D_1(u,v|\theta)&:=& \frac{\partial}{\partial u} C(u,v|\theta)
\end{eqnarray}
for $u,v \in ]0,1[$ the partial derivative of the copula with respect to the first variable. Note that this is the conditional density of the random variable $V:=F_Y(Y)$ given $U:=F_X(X)$ \citep{har97}. Table \ref{tab:der_copula} in \ref{app:copulae} shows the partial derivative \eqref{eq:parder} of the Clayton, Gumbel, Frank and Gauss copula (see e.g. \cite{aas09} or \cite{scheps12} for more details).
\begin{proposition}[Density function]\label{pro:densitymass}
The joint density function $f_{X,Y}$  of a continuous random variable $X$ and a discrete random variable $Y$
is given by
\begin{eqnarray}
f_{X,Y}(x,y|\theta)&=&f_X(x) \left(D_1(F_X(x),F_Y(y)|\theta)-D_1(F_X(x),F_Y(y-1)|\theta)\right)\,.
\label{eq:Gen_PDF}
\end{eqnarray}
\end{proposition}
\begin{proof}
By definition
\begin{eqnarray*}
\frac{\partial}{\partial x}P(X\leq x,Y=y)
				&=&\frac{\partial}{\partial x}P(X\leq x,Y\leq y)-\frac{\partial}{\partial x}P(X\leq x,Y\leq y-1)\\
				&=&\frac{\partial}{\partial x}C(F_X(x),F_Y(y)|\theta)-\frac{\partial}{\partial x}C(F_X(x),F_Y(y-1)|\theta)\\
				&=&f_X(x)\left(D_1(F_X(x),F_Y(y),\theta)-D_1(F_X(x),F_Y(y-1)|\theta)\right),
\end{eqnarray*}
which proves the statement.
\end{proof}
%-----------------------------------%
\subsection{Marginal distributions} %
%-----------------------------------%
The framework presented in the preceding subsection holds for general pairs of continuous-discrete random variables. In this paper, we focus on joint models for a positive average claim size $X$ and a positive number of claims $Y$. We  model the average claim size $X$  via a Gamma distribution
\begin{eqnarray}
f_X(x|\mu,\delta)&=&\frac{1}{x\Gamma\left(\frac{1}{\delta}\right)}\left(\frac{x}{\mu\delta}\right)^\frac{1}{\delta} \exp\left(-\frac{ x}{\mu\delta}\right)\quad\mbox{for }x>0\,,
\label{eq:Gamma_Density}
\end{eqnarray}
with mean parameter $\mu>0$ and dispersion parameter $\delta>0$. The number of claims $Y$ is a positive count variable, and is modeled as a zero-truncated Poisson (ZTP) distributed variable with parameter $\lambda>0$,
\begin{eqnarray*}
f_Y(y|\lambda)&=\frac{\lambda^y}{y!\left(1- \exp(-\lambda)\right)}\exp(-\lambda)\quad\mbox{for }y=1,2,\ldots.
\end{eqnarray*}
The generality of our approach easily allows to use other appropriate distributions such as the log-normal for claim severity or the (zero-truncated) Negative Binomial for claim frequency. The models and results presented below can be extended accordingly.
%--------------------------------------------------------------------%
\subsection{Joint copula model for average claim sizes and number of claims} %
%--------------------------------------------------------------------%
Combining the marginal distributions and the copula approach, we obtain the following general model.
\begin{definition}[Joint copula model for average claim sizes and number of claims]\label{def:jointmodel}
The copula-based Gamma and zero-truncated Poisson model for positive average claim sizes $X$ and positive number of claims $Y$ is defined by the joint density function
\begin{equation}
\begin{split}
f_{X,Y}&(x,y|\mu,\delta,\lambda,\theta)\\
& = f_X(x|\mu,\delta)  \left(D_1(F_X(x|\mu,\delta),F_Y(y|\lambda)|\theta)-D_1(F_X(x|\mu,\delta),F_Y(y-1|\lambda)|\theta)\right)\,,
\end{split}
\label{eq:jointmodel}
\end{equation}
for $x>0$ and $y=1,2,\ldots$.
\end{definition}

The model depends on four parameters: the parameters  $\mu,\delta$ (Gamma) and $\lambda$ (ZTP) for the marginal distributions, and the copula parameter $\theta$. Table \ref{tab:parameters} displays the parameters and their relationships to the joint distribution.

\begin{table}[t]
\begin{center}
\begin{tabular}{rccc}
\hline
&average &&\\
& claim size $X$ & number of claims $Y$ & copula family \\
\hline
distribution&Gamma&zero-truncated Poisson& Gauss, Clayton\\
&&& Gumbel, Frank\\
parameter(s) &$\mu>0,\,\delta>0$& $\lambda >0$& $\theta \in \Theta$\\
expectation & $E(X)=\mu$ & $E(Y)=\frac{\lambda}{1- e^{-\lambda}}$& --- \\
variance & $Var(X)=\mu^2 \delta$&$Var(Y)=\frac{\lambda(1-e^{-\lambda}(\lambda+1))}{(1- e^{-\lambda})^2}$ & ---
\end{tabular}
\end{center}
\caption{Model parameters of the joint distribution for average claim sizes $X$ and number of claims $Y$. The definition of the copula families is provided  in \ref{app:copulae}. }
\label{tab:parameters}
\end{table}

We now illustrate the influence of the copula parameters and families on the conditional distribution of $Y$. Therefore, we use the following result.
 \begin{proposition}\label{pro:conditional}
The conditional distribution $Y|X=x$ of the number of claims given an average claim size of $x$ under the copula-based model of Definition \ref{def:jointmodel} is given by
\begin{equation}
\begin{split}
P\left(Y=y|X=x,\mu,\delta,\lambda,\theta\right)=\ & D_1(F_X(x|\mu,\delta),F_Y(y|\lambda)|\theta)\\
& -D_1(F_X(x|\mu,\delta),F_Y(y-1|\lambda)|\theta)\,.
\end{split}
\label{eq:cond_dist}
\end{equation}
\end{proposition}
\begin{proof}
This result follows from Proposition \ref{pro:densitymass}, as by definition for two random variables $X$ and $Y$
\begin{equation*}
P(Y=y|X=x)= \frac{f_{X,Y}(x,y)}{f_X(x)}\,.
\end{equation*}
\end{proof}
\begin{example}\label{ex:parameters}
We consider a group of policy holders with an expected number of claims of $\lambda=2.5$. The average claim size is set to $\mu=1000$ Euro, and we assume that the standard deviation of the average claim size equals $300$ Euro, which leads to a dispersion parameter of
\begin{equation*}
\delta= \frac{300^2}{1000^2}=0.09\,.
\end{equation*}
We condition on an average claim size of $x=1200$ Euro.
\end{example}

\begin{figure}[t]
\begin{center}
\includegraphics[width=7.5cm]{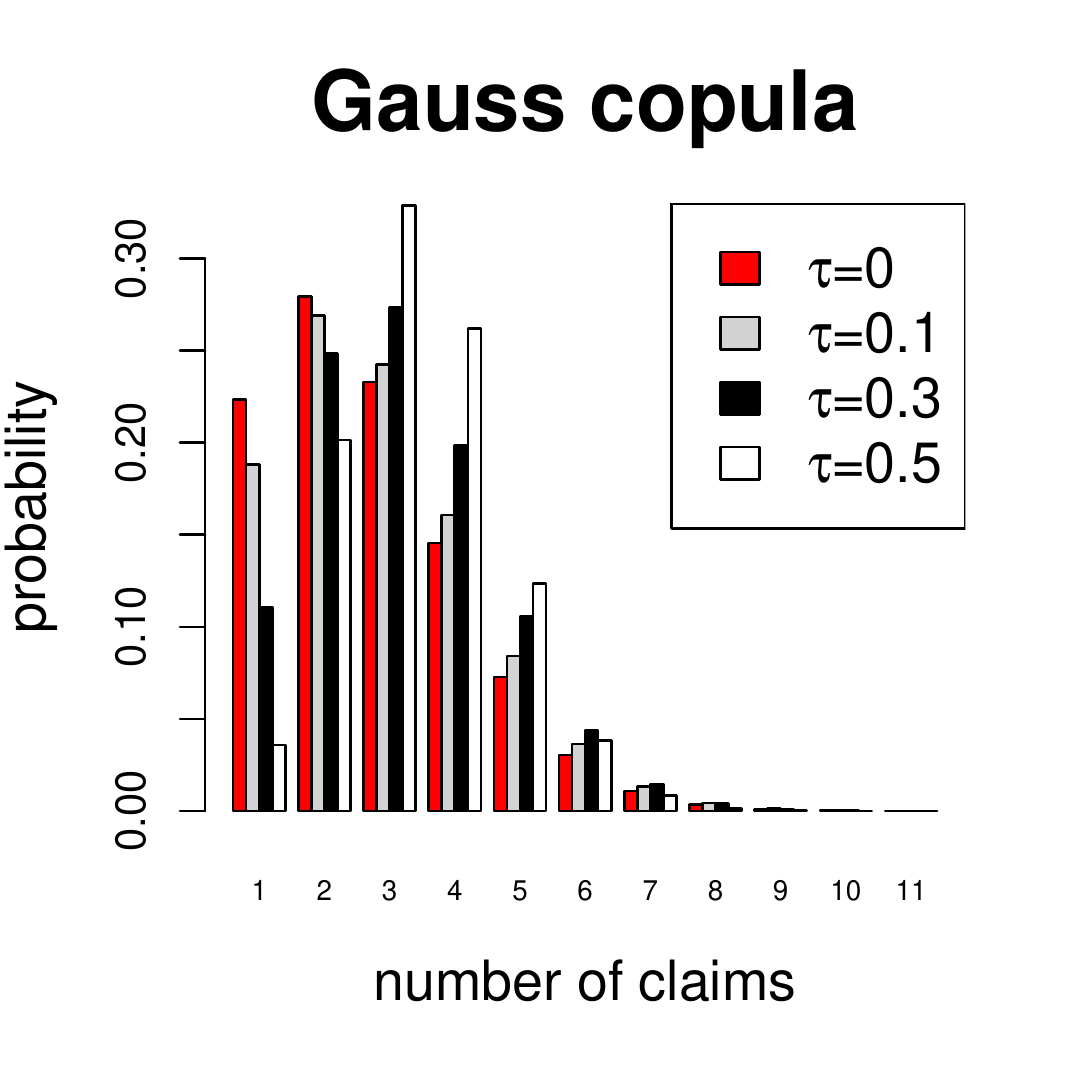}\includegraphics[width=7.5cm]{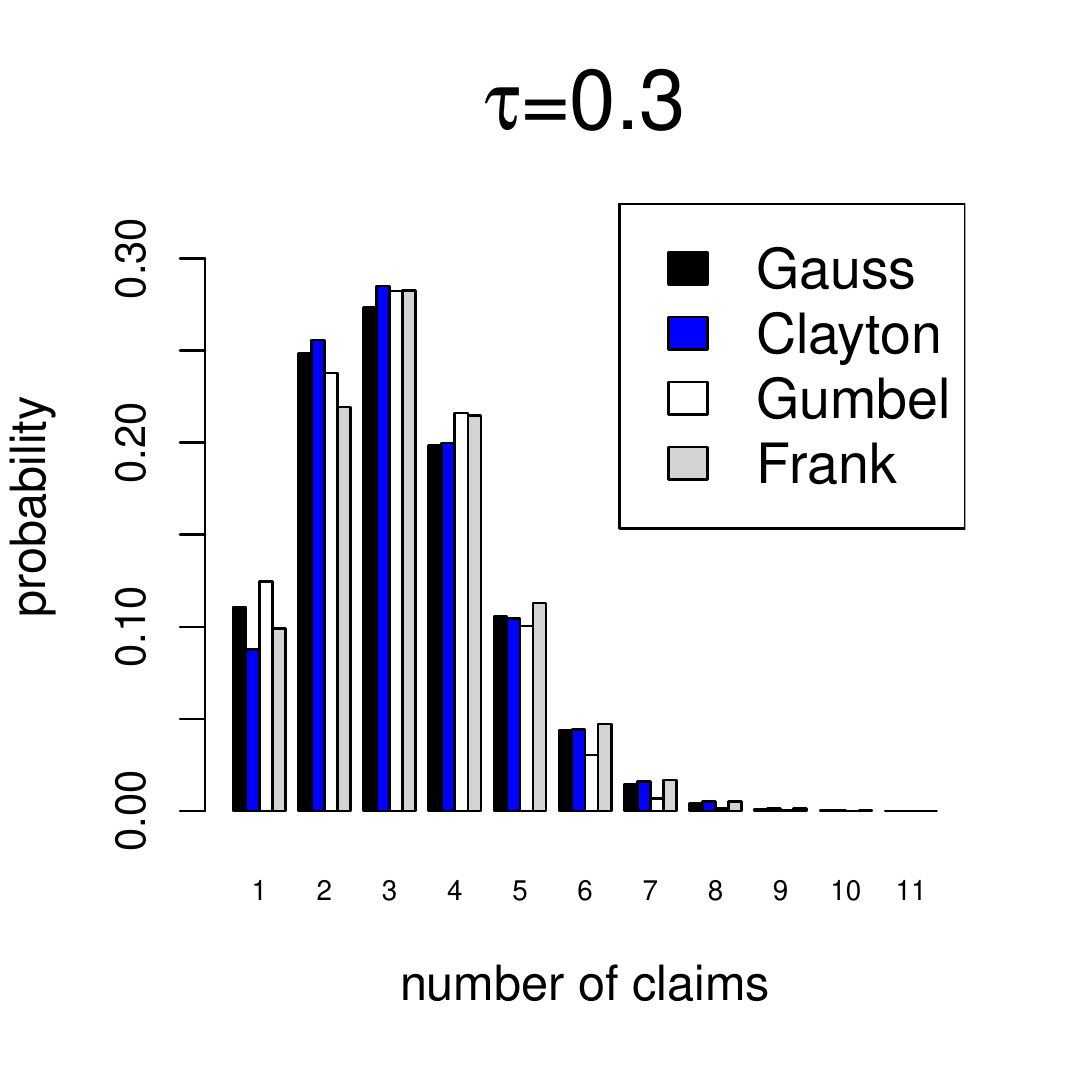}%
\end{center}
\caption{Conditional probability mass function of the number of claims $Y$. Marginal distributions: expected average claimsize $\mu=1000$ Euro with dispersion parameter $\delta=0.09$, expected number of claims $\lambda=2.5$. We condition on an average claim size of $x=1200$ Euro. Left: Gauss copula with $\tau=0;\,0.1;\,0.3;\,0.5\,$. Right: Gauss, Clayton, Gumbel and Frank copula with $\tau=1/3.$}%
\label{fig:cond_dist}%
\end{figure}
The left panel in Figure \ref{fig:cond_dist} displays the conditional probability mass function \eqref{eq:cond_dist} of $Y|X=x$ for a Gauss copula with four different values of $\tau=0;\,0.1;\,0.3;\,0.5\,$. We observe that the
 four probability mass functions are different, and that for higher values of $\tau$, more mass is assigned to larger values of $y$. This is due to the dependence of $X$ and $Y$ and the fact that the conditioning value $x=1200$ Euro is much higher than the expected average claim size of $\mu=1000$ Euro. The right panel displays the conditional probability mass function for $\tau=1/3$ and the four different copula families. The choice of the copula family clearly influences the conditional distribution. In particular, the upper-tail dependent Gumbel copula shifts the distribution to the right compared to the other copulae. This leads to a flexible class of dependence models between a discrete and a continuous variable.
%-------------------------------%
\section{Policy loss estimation} %
%-------------------------------%
 Next, we focus on the distribution of the policy loss.
\begin{definition}[Policy loss] For a policy with average claim size $X$ and number of claims $Y$,  the policy loss is defined as the product of the two quantities,
\begin{align*}
L&:= X \cdot Y \,.
\end{align*}
 \end{definition}
The policy loss is a positive, continuous random variable, and it depends on the four model parameters displayed in Table \ref{tab:parameters}. A main contribution of this paper is the following result.
\begin{theorem}\label{thm:policy}
The distribution of the policy loss $L$ is given by the density function
\begin{align*}
f_L(l|\mu,\delta,\lambda,\theta)
&=  \sum_{y=1} ^\infty \left[  D_1\left(F_X\left(\textstyle\frac{l}{y}|\mu,\delta\right),F_Y\left(y|\lambda\right)|\theta\right) -D_1\left(F_X\left(\textstyle\frac{l}{y}|\mu,\delta\right),F_Y\left(y-1|\lambda\right)|\theta\right)\right]\\
&\quad \cdot \frac{1}{y} f_X\left(\textstyle\frac{l}{y} \left|\mu,\delta\right. \right)
\end{align*}
for $l>0$.
\end{theorem}
\begin{proof}
For simplicity of notation, we omit the model parameters from the formulas. We consider the two-dimensional random variable
\begin{align*}
\left(L,Y\right)^\top &\in \mathbb{R}^+ \times \{1,2,\ldots\}\,
\end{align*}
and derive its joint density mass function. By definition (see Equation \eqref{eq:density_mass})),
\begin{align*}
f_{L,Y}(l,y)&=\frac{\partial}{\partial l} P(L\leq l,Y=y)\\
&= \frac{\partial}{\partial l} P\left(X\leq \textstyle\frac{l}{y},Y=y\right)
\end{align*}
as $X=L/Y$. Substituting $x=l/y$, we obtain
\begin{align*}
f_{L,Y}(l,y)&=\frac{\partial}{\partial x} P\left(X\leq x,Y=y\right)\cdot \frac{\partial x}{\partial l}\\
&= f_{X,Y}\left(\textstyle\frac{l}{y},y\right) \cdot \frac{1}{y}\,.
\end{align*}
The result then follows by marginalizing over the discrete random variable $Y$.
\end{proof}

This implies that we can evaluate the density of the policy loss for all of our four copula models, given a fixed set of parameters. Further, we can evaluate its mean, variance and quantiles based on the density function.

In a first step, we visualize the densities for a given set of parameters in order to investigate the differences between the four copula types and the degree of dependence between the average claim size and the average number of claims. In the simulation study (Section \ref{sec:simu}) and the case study (Section \ref{sec:real}), we show that in the context of regression, the copula-based model leads to a more precise estimation of the policy loss compared to the independence assumption.

We continue Example \ref{ex:parameters} and use the same parameter settings for the marginal distributions. Figure \ref{fig:total_loss}  displays the  density of the policy loss for all four copula families and for three different values of Kendall's $\tau$ equal to $0.1$, $0.3$ and $0.5$.  First, we observe that the distribution is in general left skewed. Further, we observe that the theoretical densities tend to be multimodal, and the multiple modes become more distinct for increasing values of Kendall's $\tau$. The skewness and multi-modality can be readily explained by Theorem \ref{thm:policy}. Setting
\begin{align*}
\omega(y,l|\mu,\delta,\lambda,\theta)&:= \frac{1}{y}P\left(Y=y\left|X=\frac{l}{y},\mu,\delta,\lambda,\theta\right.\right)\,,
\end{align*}
the density of the policy loss can be written as an infinite mixture of Gamma distributions
\begin{align*}
f_L(l|\mu,\delta,\lambda,\theta)&= \sum_{y=1} ^\infty \omega(y,l|\mu,\delta,\lambda,\theta) \cdot f_X\left(\frac{l}{y} \left|\mu,\delta \right.\right) \,.
\end{align*}
As the individual Gamma densities are skewed, the density of the mixture tends to be skewed, too. Moreover, a mixture of unimodal Gamma densities can be multi-modal as well. The parameter settings of the model influence the number of the modes and how pronounced they are.

\begin{figure}[t]
\begin{center}
\includegraphics[width=3.5cm]{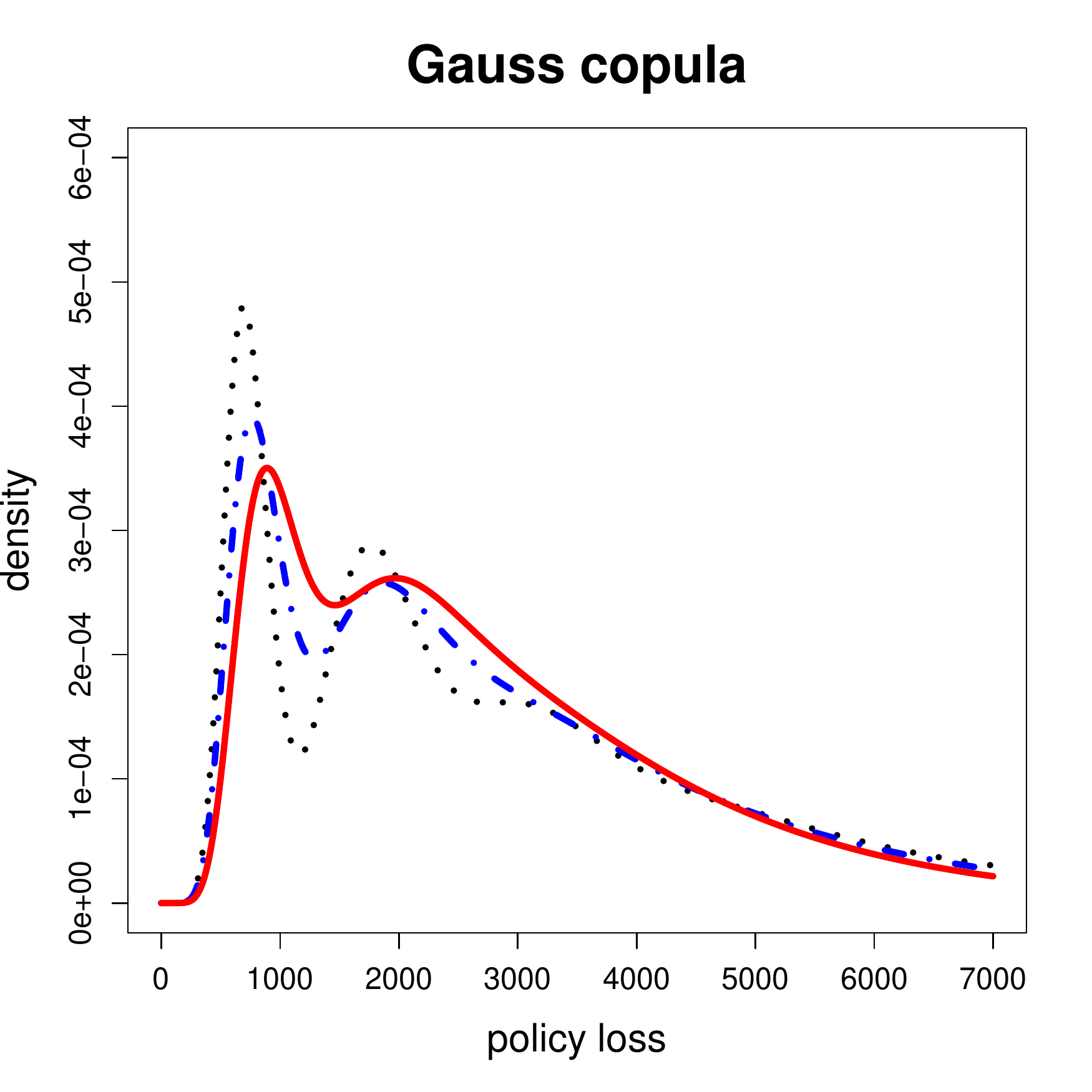}\includegraphics[width=3.5cm]{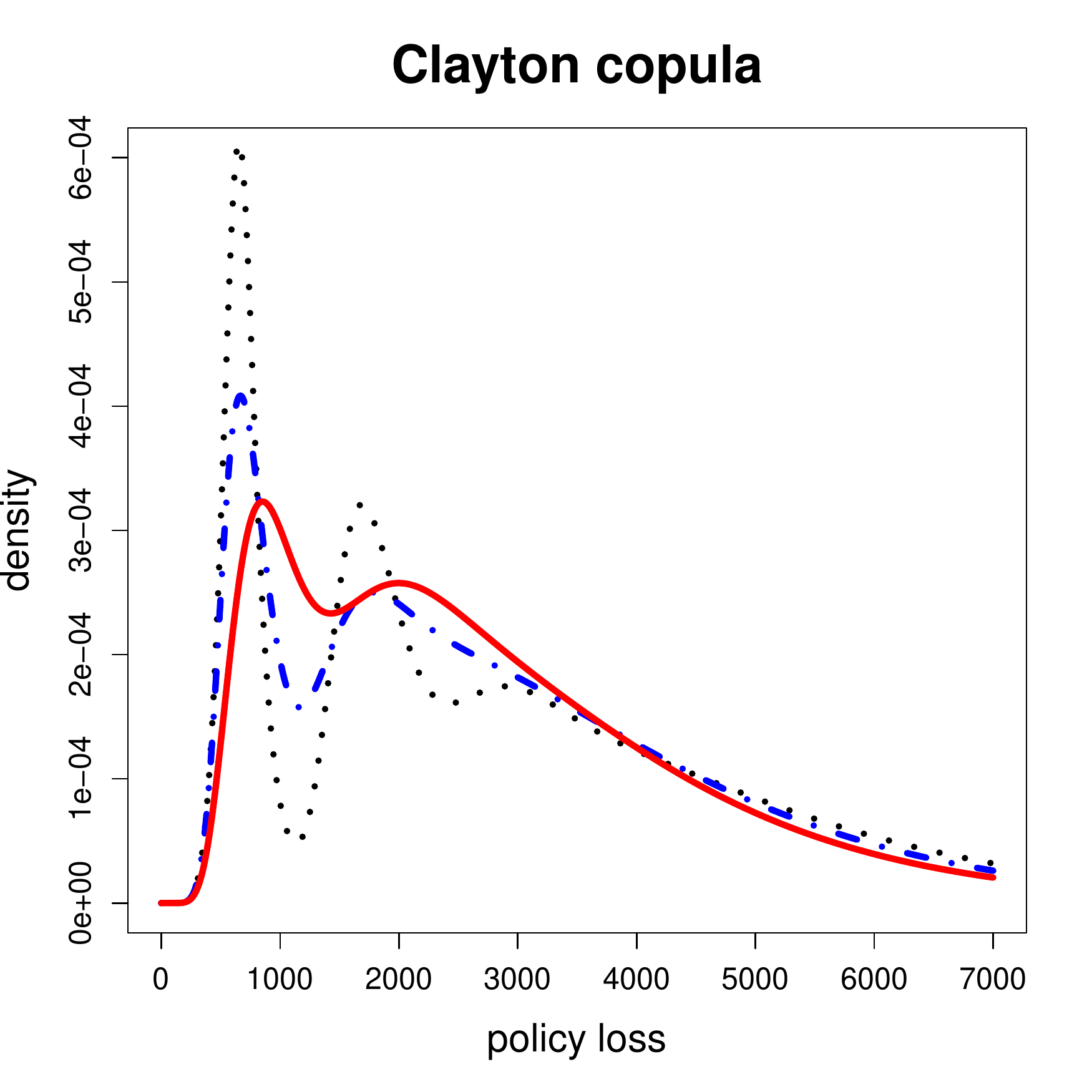}\includegraphics[width=3.5cm]{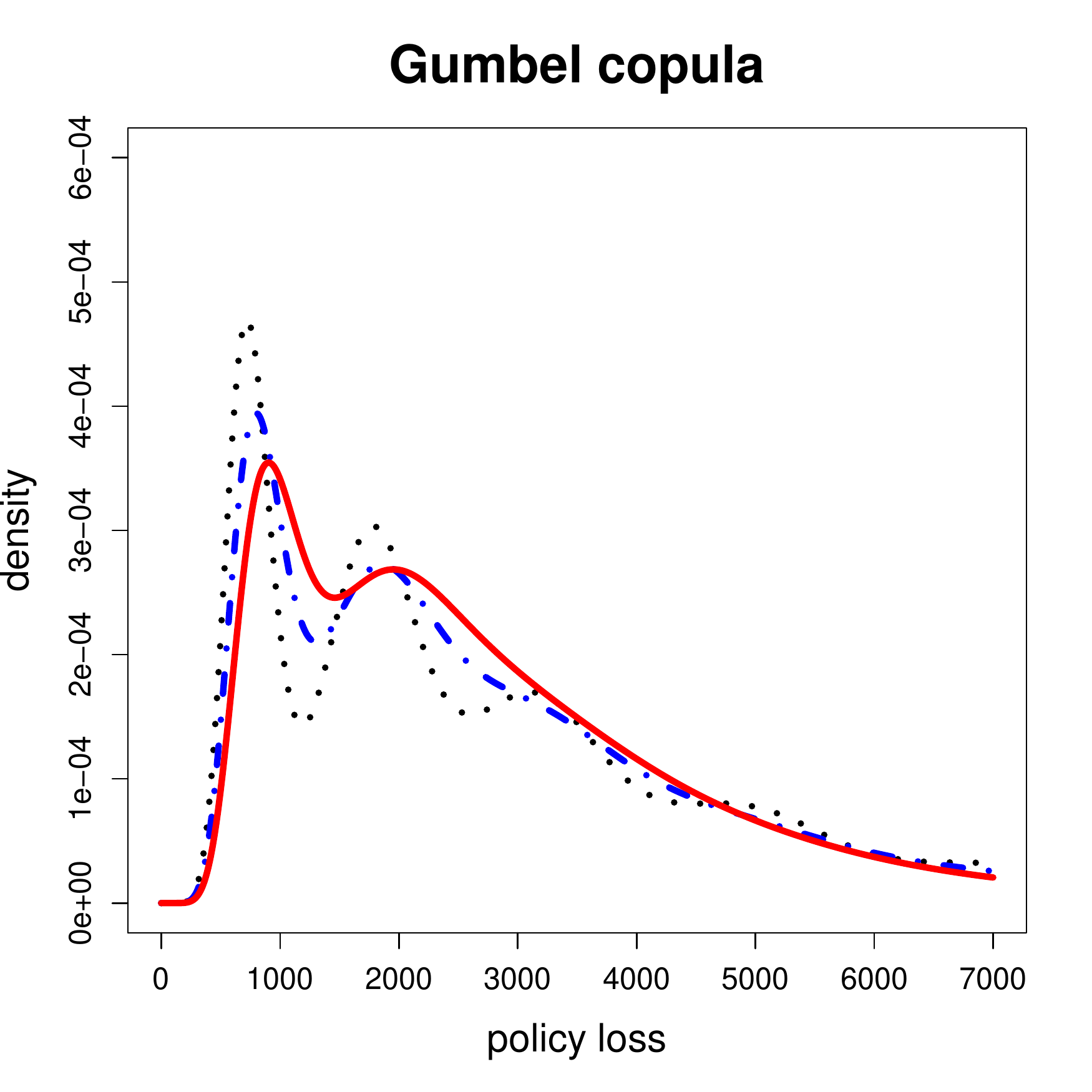}\includegraphics[width=3.5cm]{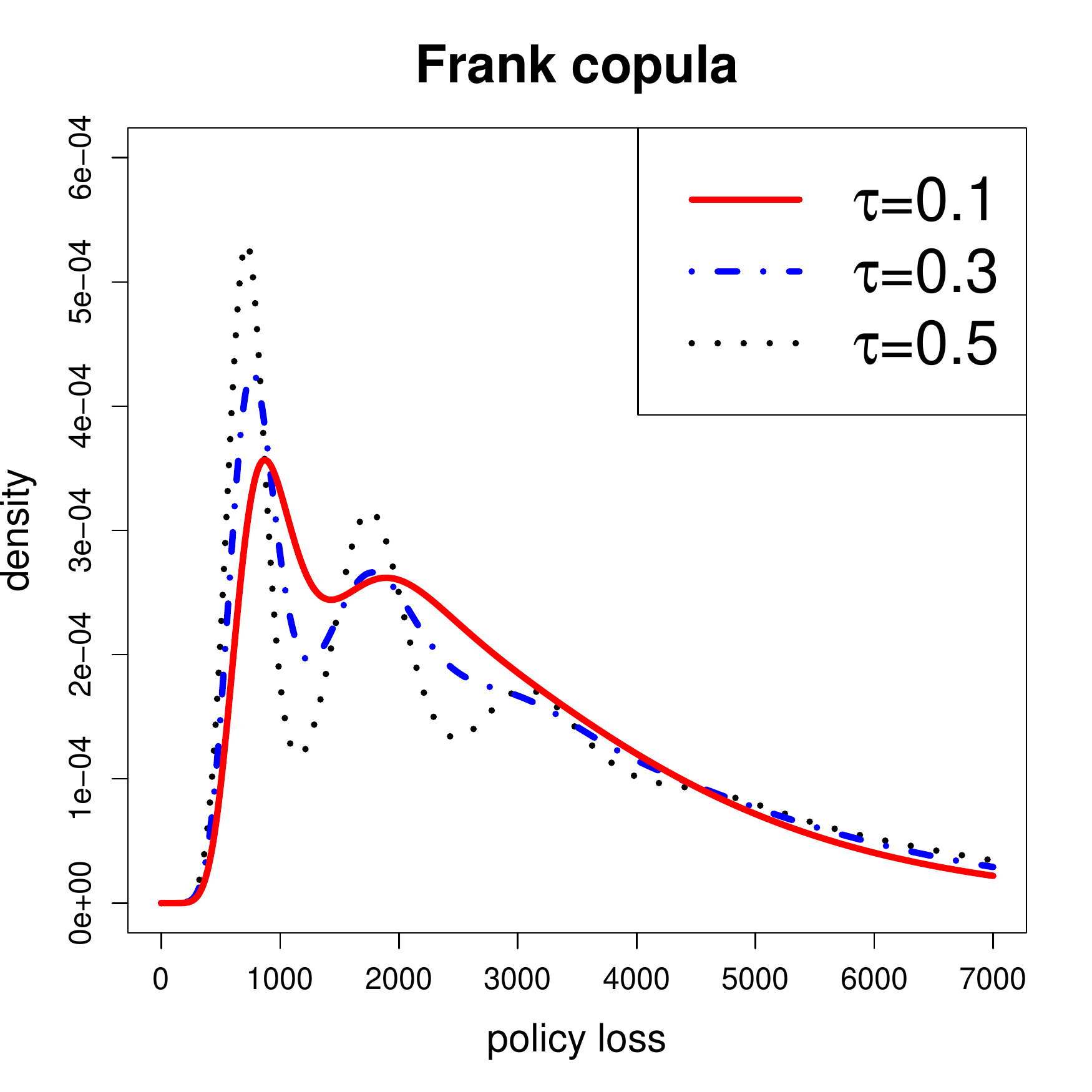}%
\end{center}
\caption{Densities of the policy loss for the four copula families and three different values of Kendall's $\tau$. We use the parameters settings of Example \ref{ex:parameters}.}%
\label{fig:total_loss}%
\end{figure}
Figure \ref{fig:expected_policy_loss} displays the expectation  $\mu_L$, the $25\%$-quantile $q_{0.25;L}$ and $75\%$-quantile $q_{0.75;L}$ of the policy loss as a
function of Kendall's $\tau$. All three quantities are evaluated using  numerical integration and numerical root solvers. We use the parameter settings of Example \ref{ex:parameters} for the marginal distributions. The solid and dotted lines indicate the mean and the quantiles if we assume that average claim sizes and number of claims are independent.  We observe that the independence assumption leads to an overestimation of the policy loss if average claim sizes and number of claims have a negative monotone association (i.e. $\tau<0$), and it leads to an underestimation if $\tau>0$. As an example, we compare the expected policy loss under independence (which equals $2723$ Euro) to the expected policy loss for $\tau=0.2$. We obtain $2860\, (+5\%)$ Euro (Gauss), $2837\, (+4\%)$ Euro (Clayton),  $2880\,(+6\%)$ Euro (Gumbel) and  $2850\,(+5\%)$ Euro (Frank). %(The percentage in brackets denotes the relative increase with respect to the expected policy loss under independence.)

Based on Figures \ref{fig:total_loss} and \ref{fig:expected_policy_loss}, we observe a strong dependence of the distribution of the policy loss on the size of Kendall's $\tau$. However, we do not observe a strong dependence on the choice of the copula family.

\begin{figure}[hb]
\begin{center}
\includegraphics[width=3.5cm]{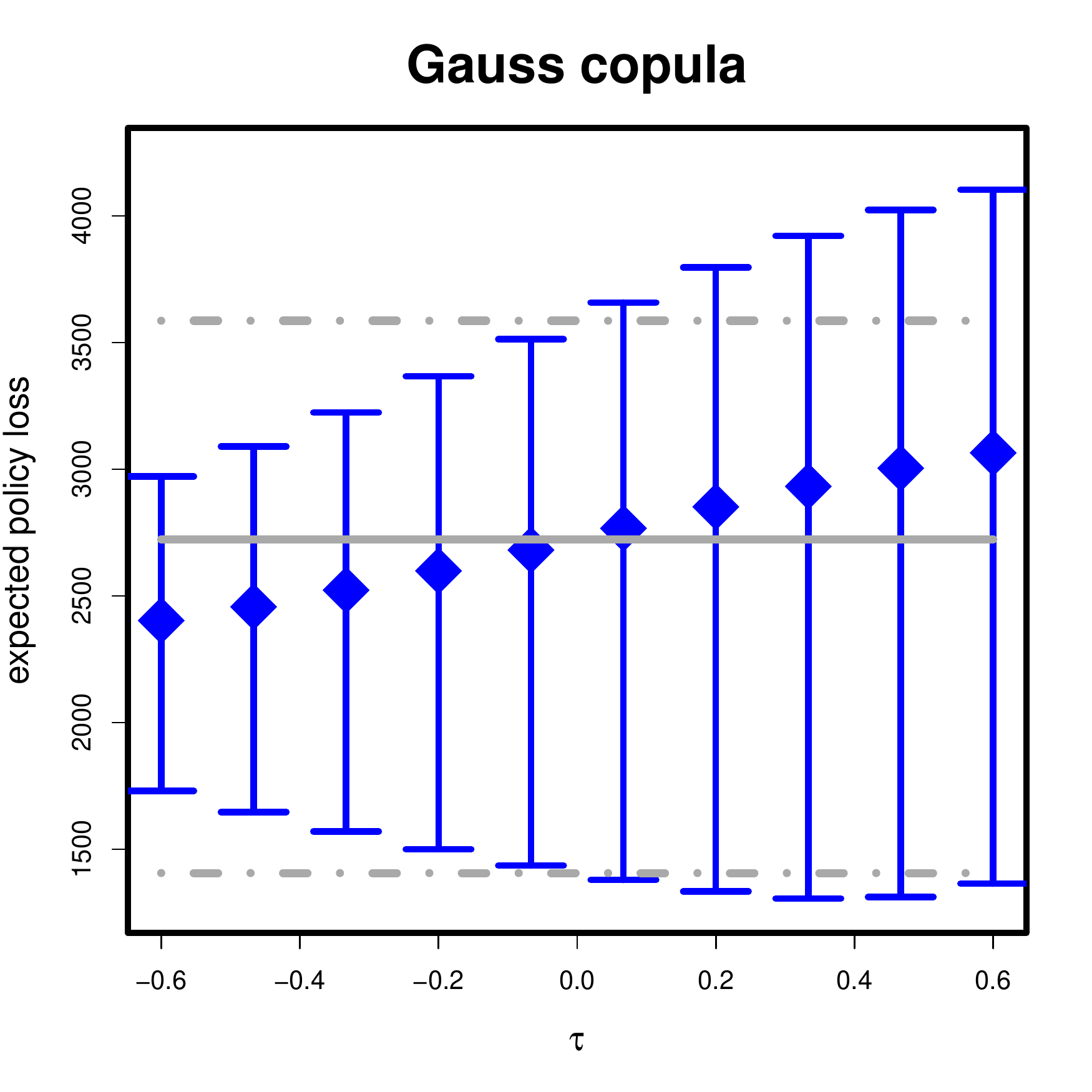}\includegraphics[width=3.5cm]{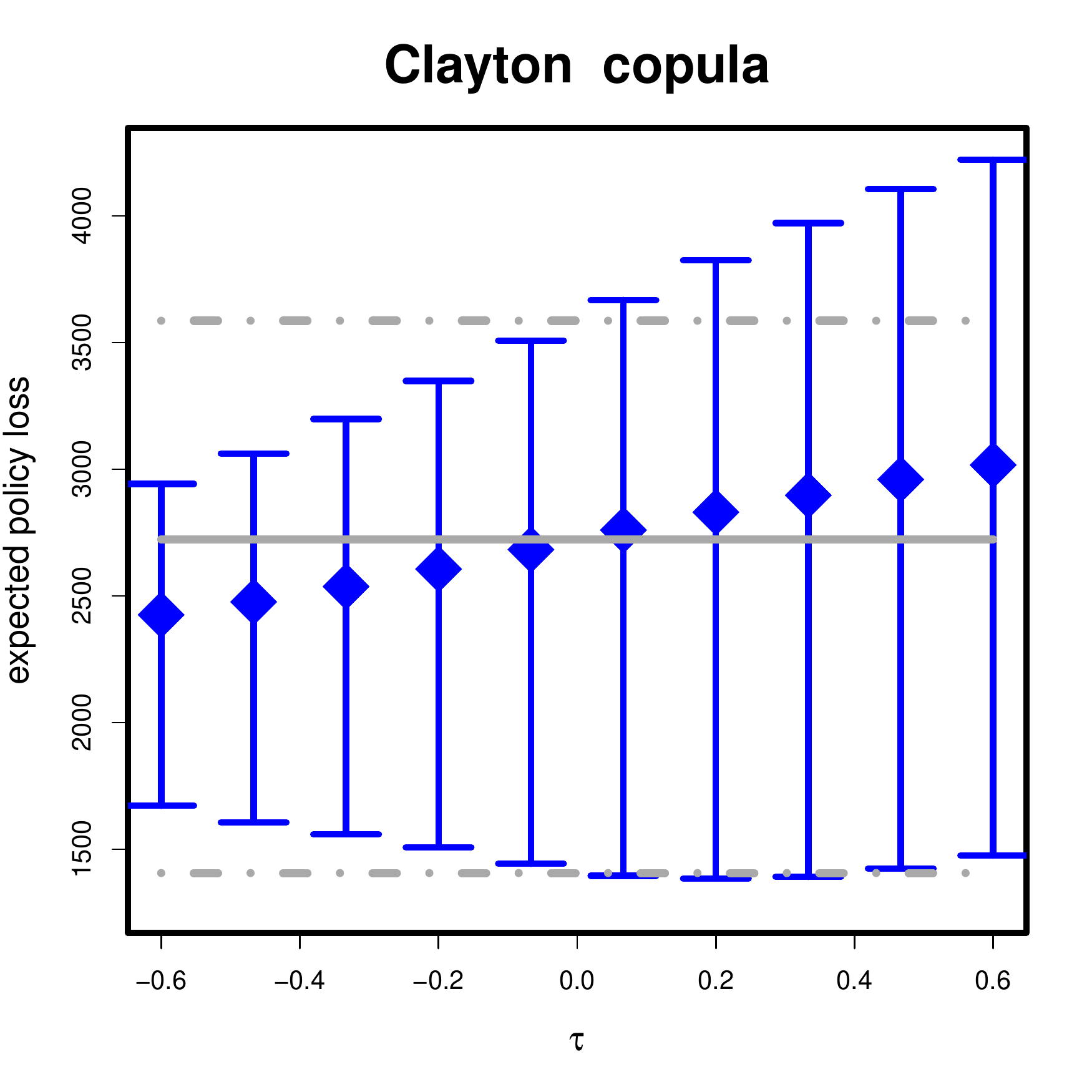}\includegraphics[width=3.3cm]{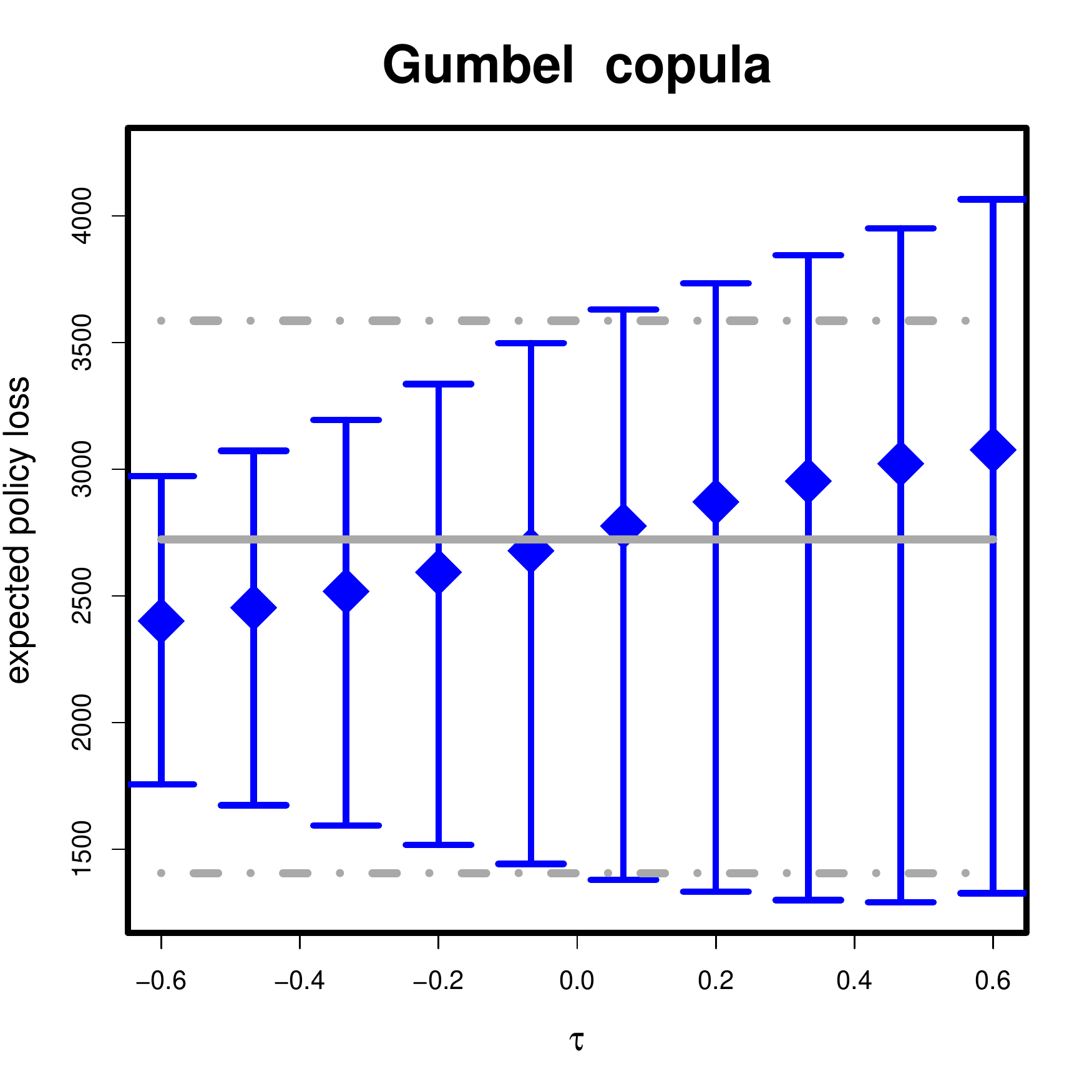}\includegraphics[width=3.5cm]{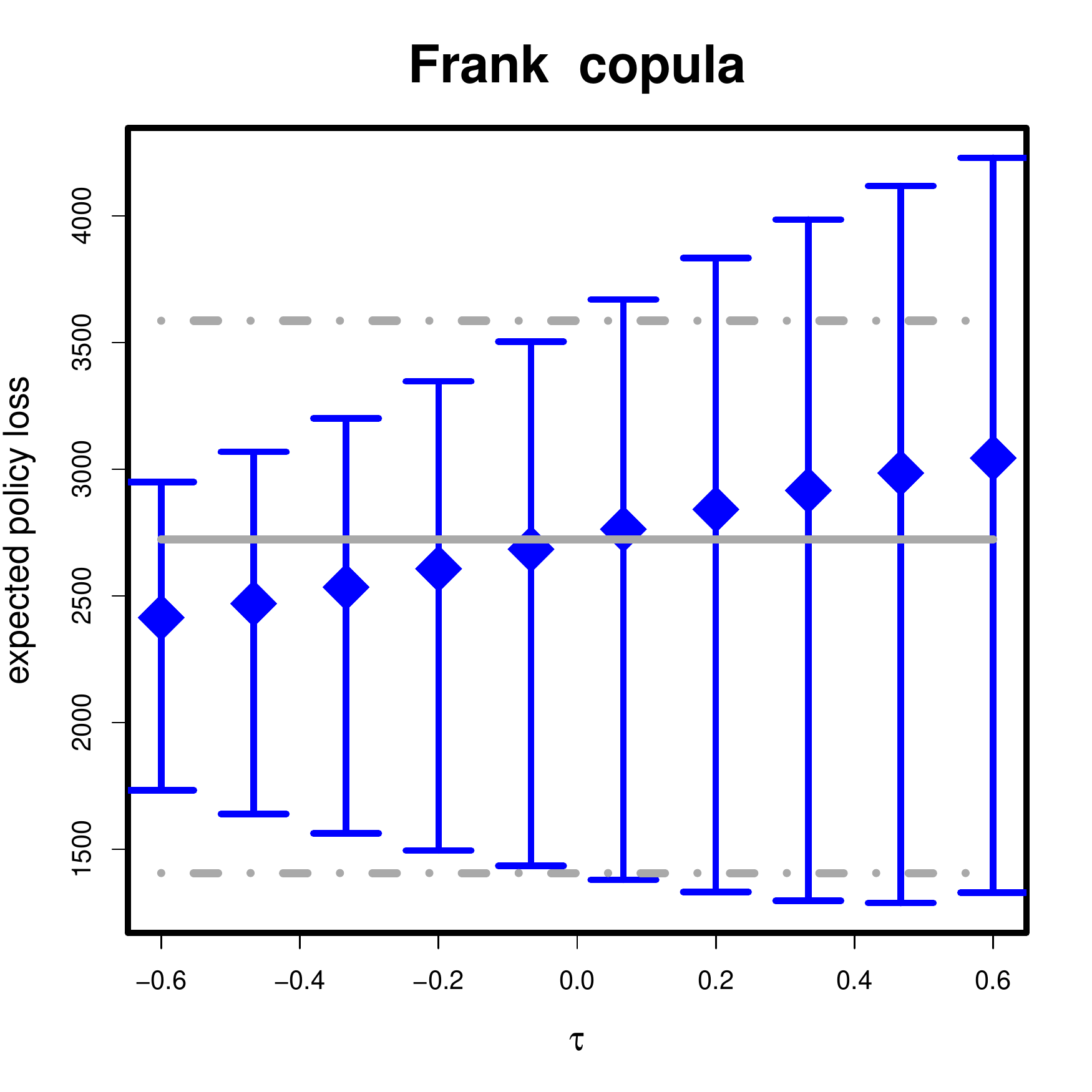}
\end{center}
\caption{Expected policy loss (blue diamonds) and upper and lower quartiles for the four copula families, seen as a function of Kendall's $\tau$. For negative values of Kendall's $\tau$ the Clayton and the Gumbel copula have been rotated. The parameter settings for the marginal distributions are taken from Example \ref{ex:parameters}. The grey solid and dotted lines indicate the expected policy loss and upper/lower quartiles if we assume independence.}%
\label{fig:expected_policy_loss}
\end{figure}
%----------------------------------------------------------------------%
\section{Copula regression model for average claim sizes and number of claims} %
%----------------------------------------------------------------------%
In the two previous sections, we modeled average claim sizes and number of claims independently of possible covariates. In order to incorporate covariates, we use the approach by \cite{czakas10}. We extend the joint model \eqref{eq:jointmodel} for average claim sizes $X$ and number of claims $Y$ by allowing the marginal distributions of $X$ and $Y$ to depend on a set of covariates. More precisely, we apply generalized linear models for the marginal regression problems and combine these with bivariate copula families.
\subsection{Model formulation}
Let $X_i\in\mathbb{R}_{+}$, $i=1,2,\ldots,n,$ be independent continuous random variables and let $Y_i\in\mathbb{N}_{>0}$, $i=1,2,\ldots,n,$ be independent discrete random variables. We model $X_i$ in terms of a covariate vector $\bm r_i\in\mathbb{R}^{p}$ and $Y_i$ in terms of a covariate vector $\bm s_i\in\mathbb{R}^q$. The marginal regression models are specified via
\begin{align*}
X_i&\sim \text{Gamma}(\mu_i,\delta)\qquad\mbox{with }\ln(\mu_i)={\bm r_i} ^\top \bm\alpha,\\
Y_i&\sim \text{ZTP}(\lambda_{i})\qquad\mbox{with } \ln(\lambda_{i})=\ln(e_i)+{\bm s_i}^\top\bm\beta.
\end{align*}
Here $e_i$ denotes the exposure time. We remark that the covariate vectors $\bm r_i$ and $\bm s_i$ can be distinct.

%Further, we define the design matrices for $X_1,\ldots,X_n$ and $Y_1,\ldots,Y_n$ as

%---------------------------------%
\subsection{Parameter estimation} %
%---------------------------------%
We need to estimate the unknown parameter vector
\begin{align} \label{eq:par}
\bm \upsilon:=(\bm\alpha^\top,\bm\beta ^\top,\theta,\delta)^\top \in \mathbb{R}^{p+q+2}
\end{align}
based on $n$ observation pairs $(x_i,y_i)$. Here, we use maximum-likelihood estimation techniques. By definition, the loglikelihood of the model parameters \eqref{eq:par} is
\begin{align}
\label{eq:loglik}
\ell\left( \bm \upsilon|\bm x,\bm y \right)&= \sum_{i=1} ^n \ln\left( f_{X,Y}(x_i,y_i|{\bm \upsilon})\right)
\end{align}
with
\begin{align*}
{\bm x}= \left(x_1,\ldots,x_n\right)^\top \in \mathbb{R}^n &\text{ and } {\bm y}= \left(y_1,\ldots,y_n\right)^\top \in \mathbb{R}^n\,.
\end{align*}
The maximum likelihood estimates are given by
\begin{align*}
\widehat {\bm \upsilon}&=\text{arg}\max_{ \bm \upsilon} \ell\left( \bm \upsilon|\bm x,\bm y \right)\,.
\end{align*}
In general, there is no closed-form solution. Therefore, we have to maximize the loglikelihood numerically. In this paper, we apply the BFGS optimization algorithm (a quasi Newton method) to maximize the loglikelihood \eqref{eq:loglik}. As the copula parameter $\theta \in \Theta$ is in general restricted (see  \ref{app:copulae}), we transform $\theta$ via a function $g:\Theta \rightarrow \mathbb{R}$ such that $g(\theta)$ is unrestricted. As an example, for the Gauss copula, the copula parameter $\theta$ lies in $]-1,1[$, and the transformation is defined as
\begin{eqnarray*}
g(\theta)&=&\frac{1}{2} \ln \left(\frac{1+\theta}{1-\theta}\right)\,.
\end{eqnarray*}

We then optimize the logliklihood with respect to $(\bm\alpha^\top,\bm\beta ^\top,g(\theta),\delta)^\top$.

Alternatively, we can  estimate the model parameters by applying the inference-for-margins (IFM) principle \citep{Joe96}. Here, we proceed in two steps. First, we estimate the marginal regression models for average claim sizes and number of claims via maximum-likelihood. We obtain estimates
\begin{align*}
\widehat{\bm\mu}&= \exp\left({\bm R} \widehat{\bm \alpha}\right)\in \mathbb{R}^n\\
\widehat{\bm\lambda}&= \exp\left({\bm S} \widehat{\bm \beta}\right)\odot {\bm e}\in \mathbb{R}^n
\end{align*}
for each observation and an estimate $\widehat{\delta}\in \mathbb{R}$ for the dispersion parameter. Here, ${\bm e} \in \mathbb{R}$ is the vector of exposure times, and $\odot$ denotes an element-wise multiplication of two vectors. These estimates are used to transform the observations ${\bm x}$ and ${\bm y}$ to
\begin{align*}
u_i&:= F_X\left(x_i|\widehat \mu_i,\widehat \delta\right) \in [0,1]\\
v_i&:= F_Y\left(y_i|\widehat \lambda_i\right) \in [0,1]\\
w_i&:= F_Y\left(y_i-1|\widehat \lambda_i\right) \in [0,1].
\end{align*}
Here, $F_X$ and $F_Y$ are the distribution functions of a Gamma and zero-truncated Poisson variable respectively. In the second step, we optimize the copula parameter $\theta$ by maximizing the loglikelihood
\begin{align*}
\widetilde \ell\left(\theta|\bm u,\bm v\right):= \sum_{i=1} ^n \ln\left(D_1(u_i,v_i|\theta) -  D_1(u_i,w_i|\theta)  \right)\,.
\end{align*}
The function $\widetilde \ell$ can be maximized numerically. In general, the run-time for the IFM approach is much smaller compared to the maximization of the loglikelihood \eqref{eq:loglik}. In initial simulatons, the performance of the two methods was very similar. This confirms earlier findings by \cite{DeLeonWu2011}. Therefore, in the simulations study below, we only report the results of the maximum likelihood solution, since it is asymptotically more efficient.

Finally, we note that \cite{czakas10} recently proposed an extension of the  maximization by parts algorithm \citep{son05} to estimate the regression parameters. These methods could be easily adapted to our model. In this paper, we do not pursue this approach and estimate the parameters via maximum-likelihood.

%------------------------------------------------------%

\subsection{Asymptotic distribution of the regression parameters} \label{subsec:ci}
%------------------------------------------------------%
For the construction of approximate confidence intervals, we use the Fisher information matrix that  is defined as
\begin{align*}
\mathcal{I}\left( \bm \upsilon \right)&:= E\left[  \frac{\partial \ell(\bm \upsilon|{\bm x},{\bm y})}{\partial {\bm \upsilon}}\cdot \left(\frac{\partial \ell(\bm \upsilon|{\bm x},{\bm y})}{\partial {\bm \upsilon}}\right)^\top\ \right] \in \mathbb{R}^{(p+q+2)\times (p+q+2)}\,.
\end{align*}
Under regularity conditions (see, e.g., \cite{serfling1980}) one can show that
\begin{align*}
\sqrt{n}\left(\bm \upsilon - \widehat{{\bm \upsilon}}\right)&\stackrel{D}{\longrightarrow} \mathcal{N}_{p+q+2} \left( {\bm 0},\mathcal{I}^{-1}\left( \bm \upsilon \right)\right)\,.
\end{align*}
Here, $\mathcal{N}_k$ denotes a $k$-dimensional multivariate normal distribution. For the estimation of the Fisher information, we use the fact that \citep{LehmannCasella1998}
\begin{align*}
\mathcal{I}\left( \bm \upsilon \right)=- E\left[  \frac{\partial^2 \ell(\bm \upsilon|{\bm x},{\bm y})}{\partial^2 {\bm \upsilon}} \right]\,,
\end{align*}
and use the observed Fisher information matrix
\begin{align*}
\widehat{\mathcal{I}}\left( \bm \upsilon \right):=-   \frac{\partial^2 \ell(\bm \upsilon|{\bm x},{\bm y})}{\partial^2 {\bm \upsilon}}\,.
\end{align*}
This is the Hessian matrix of the loglikelihood function. In our case, it is feasible to compute the second partial derivatives explicitly. Moreover, the BFGS optimization algorithm  returns an approximation of the Hessian matrix that is obtained via numerical derivatives. In this paper, we use this approximation to estimate standard errors for the regression coefficients.
\subsection{Selection of the copula family}
Up to now, the copula class is assumed to be fixed. For the comparison of two copula families, we propose the likelihood-ratio test for non-nested hypotheses by \cite{vuong89}. This test is appropriate as our models are non-nested, i.e. the regression model for one copula family cannot be obtained via a restriction of the regression model for the other copula family. Let us denote by ${\bm \ell}^{(1)},{\bm \ell}^{(2)} \in \mathbb{R}^n$ the vectors of pointwise loglikelihoods for a model with copula family 1 and 2 respectively. Here, we assume that both models have the same degrees of freedom, i.e. the same number of parameters. We now compute the differences of the pointwise loglikelihood as
\begin{equation*}
m_i:=\ell^{(1)}_i - \ell^{(2)}_i,\ i=1,\ldots,n\,.
\end{equation*}
Denote by
\begin{eqnarray*}
\overline{m}&=&\frac{1}{n}\sum_{i=1}^n m_i
\end{eqnarray*} the mean of the differences. The test statistic
\begin{equation}\label{eq:vuong}
T_V:= \frac{\sqrt{n}\cdot\overline{m}}{ \sqrt{\sum_{i=1}^n \left(m_i - \overline{m}\right)^2}},
\end{equation}
is asymptotically normally distributed with zero mean and unit variance. Hence, we prefer copula family 1 to copula family 2 at level $\alpha$ if
\begin{equation*}
T_V> \Phi^{-1}\left(1-\frac{\alpha}{2}\right)\,,
\end{equation*}
where $\Phi$ denotes the standard normal distribution function. If
\begin{equation*}
T_V< \Phi^{-1}\left(\frac{\alpha}{2}\right)\,,
\end{equation*}
we prefer copula family 2. Otherwise, no decision among the two copula families is possible. We note that it is possible to adjust the test if the two models have different degrees of freedom.
%--------------------------------------%
\section{Estimation of the total loss} %
%--------------------------------------%
Recall that we model the policy loss $L_i=X_i\cdot Y_i$ for each policy holder via our joint regression model. Its distribution is determined by Theorem \ref{thm:policy}. In a next step, we are interested in the distribution of the total loss over all policy holders.

\begin{definition}[Total loss] For $n$ policies with average claim sizes $X_i$ and number of claims $Y_i$ (for $i=1,\ldots,n$),  the total loss is defined as the sum of the $n$ policy losses
\begin{equation*}
T:=\sum_{i=1}^n L_i=\sum_{i=1}^n X_i\cdot Y_i\,.
\end{equation*}
 \end{definition}
Just as the individual policy losses, the total loss is a positive, continuous random variable. An application of the central limit theorem leads to the following result.
\begin{proposition}[Asymptotic distribution of the total loss]
For $n$ independent policy losses $L_1,\ldots,L_n$ with  mean $\mu_{L_i}$ and variance $\sigma_{L_i} ^2$, the asymptotic distribution of the total loss $T$ is normal. For
\begin{equation*}
\sigma_n ^2 := \sum_{i=1} ^n \sigma^2 _{L_i}
\end{equation*}
we have
\begin{equation*}
\frac{\sqrt{n}}{\sigma_n}\left( T- \sum_{i=1} ^n \mu_{L_i}\right)\stackrel{\mathcal{D}}{\longrightarrow } \mathcal{N}\left(0, 1\right)\,.
\end{equation*}
\end{proposition}
For the estimation of the total loss, we need to estimate  the means  $\mu_{L_i}$ and variances $\sigma_{L_i} ^2$ of the individual policy losses. This is done by replacing the distribution parameters $\mu_i,\delta,\theta,\lambda_i$ of $L_i$ by their estimates obtained from our joint regression model. Then, the mean and variance can be estimated numerically.
%--------------------------%
\section{Simulation study} %
%--------------------------%
\label{sec:simu}
We consider a regression problem with $n=500$ policy groups and the following covariates: age, gender and type of car (A, B or C).  We assume that all policy groups contain the same number of persons, which leads to a constant offset. The first column of the design matrices
\begin{equation*}
\bm S:= \bm R:=(\bm r_1,\ldots,\bm r_n) ^\top \in \mathbb{R}^{500\times 5}
\end{equation*}
consists of $1$'s. This corresponds to marginal regression models with an intercept. The second column corresponds to the covariate age, and is drawn uniformly in the range of $18$ and $65$. The third column is the dummy variable corresponding to  female. Here, the probability of a female policy group is set to $1/2$. The last two columns are the two dummy variables corresponding to car type B and car type C. Car type A is represented by the intercept. The probability of a certain car type is set to $1/3$. The vector of regression coefficients are defined in Table \ref{tab:coef}. As an example, in this simulation scenario, a female driver has a negative effect on the average claim size, and a positive effect on the number of claims.
\begin{table}[t]
\begin{center}
\begin{tabular}{rccccc}
\hline
&intercept &age&female&car type B &car type C\\
\hline
average claim size $X$ &$-0.50$&$-0.05$&$-1.00$&$+2.00$&$-0.50$\\
number of claims $Y$ &$-1.00$&$+0.04$&$+0.30$&$+0.30$&$+0.20$
\end{tabular}
\end{center}
\caption{Regression coefficients for the simulation study.}
\label{tab:coef}
\end{table}
We set the constant dispersion parameter $\delta$ of the Gamma distribution to $\delta=0.25$, which implies that the coefficient of variation (CV) fulfills
\begin{equation*}
\text{CV}= \frac{\sqrt{Var(X_i)}}{E(X_i)}= \sqrt{\delta}=\frac{1}{2}\,.
\end{equation*}
We consider the four copula families and three different values of Kendall's $\tau$,
\begin{equation*}
\tau=0.1;\,0.3;\,0.5\,.
\end{equation*}
For each parameter setting, we sample $n=500$ observations from the true copula regression model, and then fit the regression coefficients and Kendall's $\tau$ via maximum likelihood. We consider estimation methods: (1) the independent model: we fit the two marginal regression models and set $\tau=0$. (2) the joint, copula-based model.  We also compute the estimated loss for each of the $n$ policies. We repeat this procedure $R=50$ times. To evaluate the performance of the two approaches, we consider the following measures for the estimated regression coefficients and the expected policy loss. For a parameter vector $\bm \gamma \in \mathbb{R}^k$ with estimate $\widehat{\bm \gamma}$, we are interested in the relative mean squared error which is defined as
\begin{equation}\label{eq:relmse}
\text{MSE}_{\text{rel}} := E\left( \frac{1}{k}\sum_{i=1} ^k \left(\frac{\gamma_i -\widehat{\gamma} _i}{\gamma_i}  \right)^2\right)\,.
\end{equation}
In the  $r$th iteration step, we obtain an estimate of \eqref{eq:relmse} via
\begin{equation*}
\widehat{\text{MSE}}_{\text{rel}} ^{(r)}:= \frac{1}{k}\sum_{i=1} ^k \left(\frac{\gamma_i -\widehat{\gamma}^{(r)}_i}{\gamma_i}  \right)^2\,.
\end{equation*}
Here,  $\widehat{\bm \gamma}^{(r)}$ is the estimate of ${\bm \gamma}$ obtained in the $r$th step. In the simulation study, we compare the mean relative mean squared error
\begin{eqnarray*}\label{eq:mean}
\overline{\text{MSE}}_{\text{rel}} &=& \frac{1}{R} \sum_{r=1} ^R \widehat{\text{MSE}}_{\text{rel}} ^{(r)}
 \end{eqnarray*}
computed over all $R$ simulation runs. Note that its variance can be estimated via
\begin{eqnarray*}\label{eq:var}
S^2_{{\overline{\text{MSE}}}_{\text{rel}}} &=& \frac{1}{R} \cdot \frac{1}{R-1} \sum_{r=1} ^R \left(\widehat{\text{MSE}}_{\text{rel}} ^{(r)} -  \overline{\text{MSE}}_{\text{rel}} \right)^2\,.
\end{eqnarray*}

Further, we investigate the size of the estimated $\tau$, the estimated total loss, and the value of the Akaike information criterion
\begin{align*}
\text{AIC}&:= - 2 \ell\left(\widehat{\bm \upsilon}|{\bm x},{\bm y}\right) + 2 \text{DoF}\,,
\end{align*}
where the Degrees of Freedom (\text{DoF}) are the number of estimated parameters in the model. Note that we have $p+q+2=12$ Degrees of Freedom for the joint model and $p+q+1=11$ Degrees of Freedom for the independence model. We prefer the model with the lower AIC score.

\begin{figure}%
\begin{center}
\includegraphics[width=5cm]{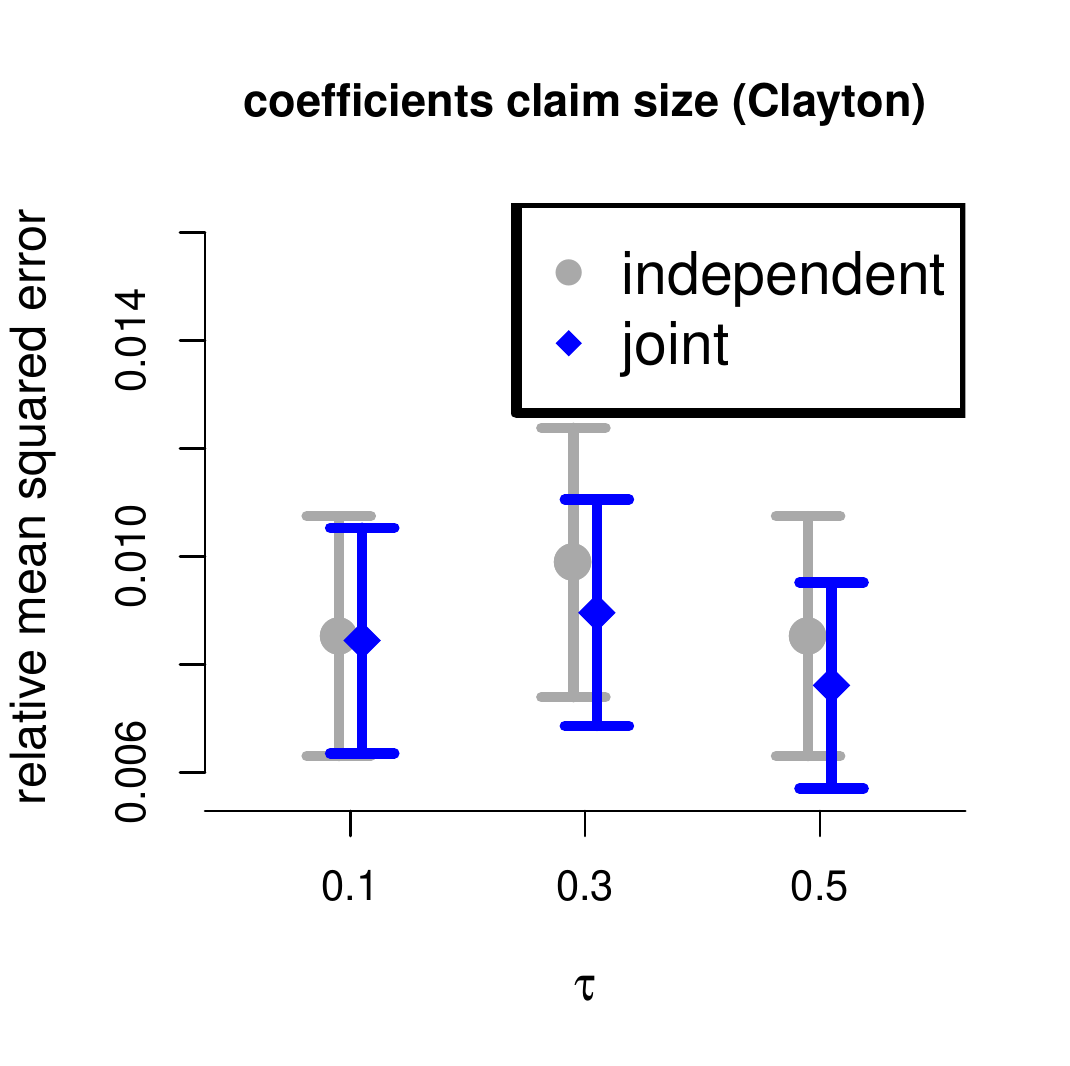}\includegraphics[width=5cm]{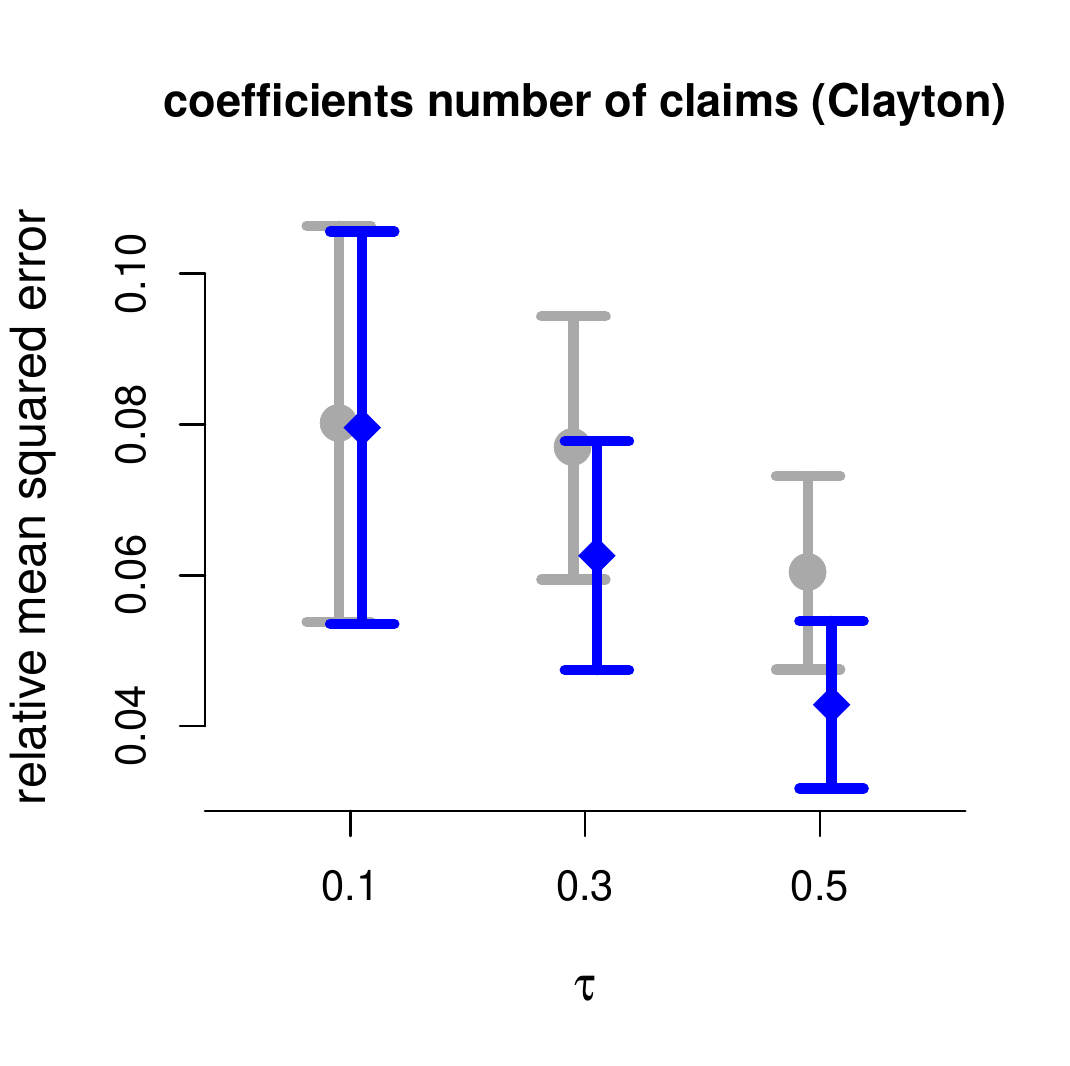}\includegraphics[width=5cm]{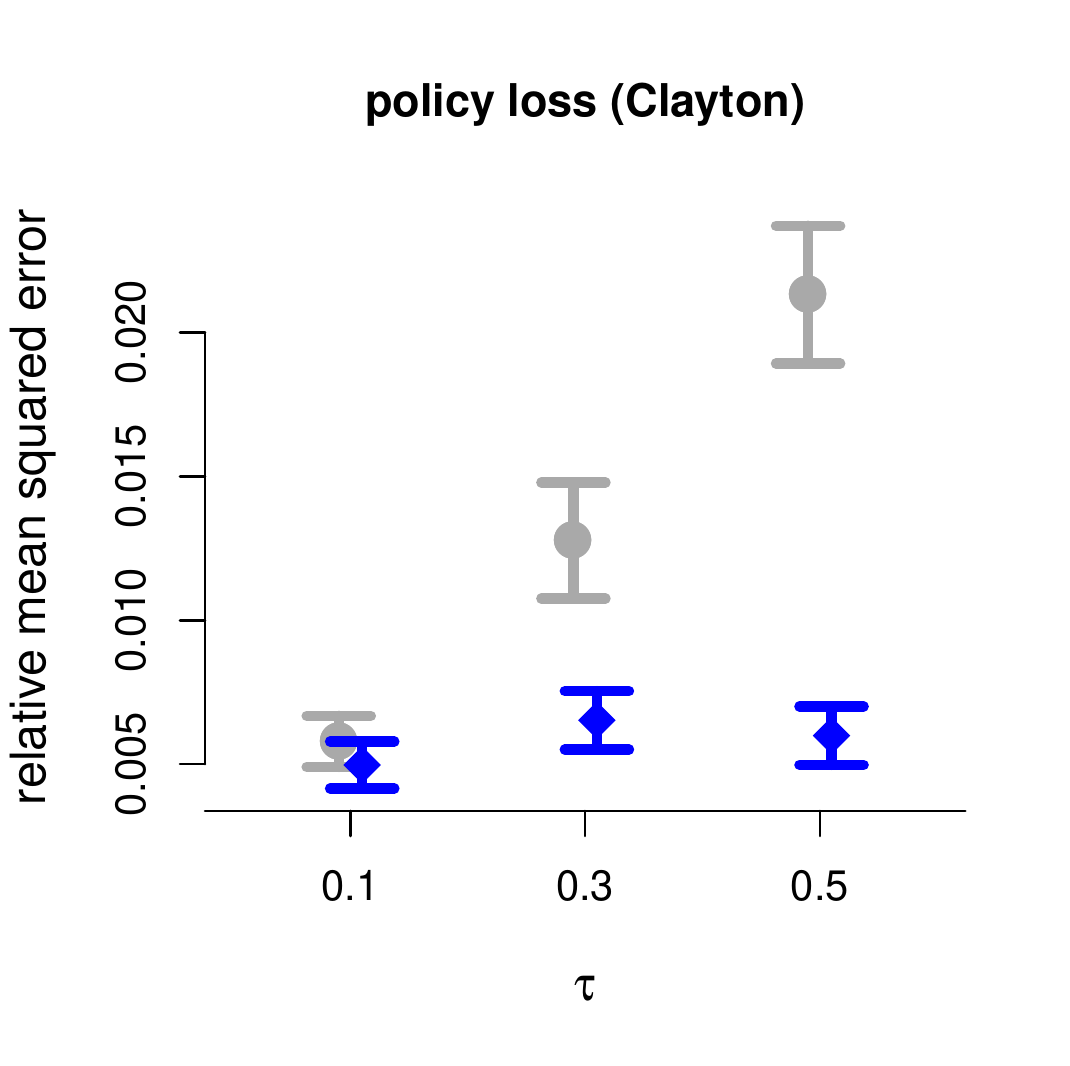}\\
\includegraphics[width=5cm]{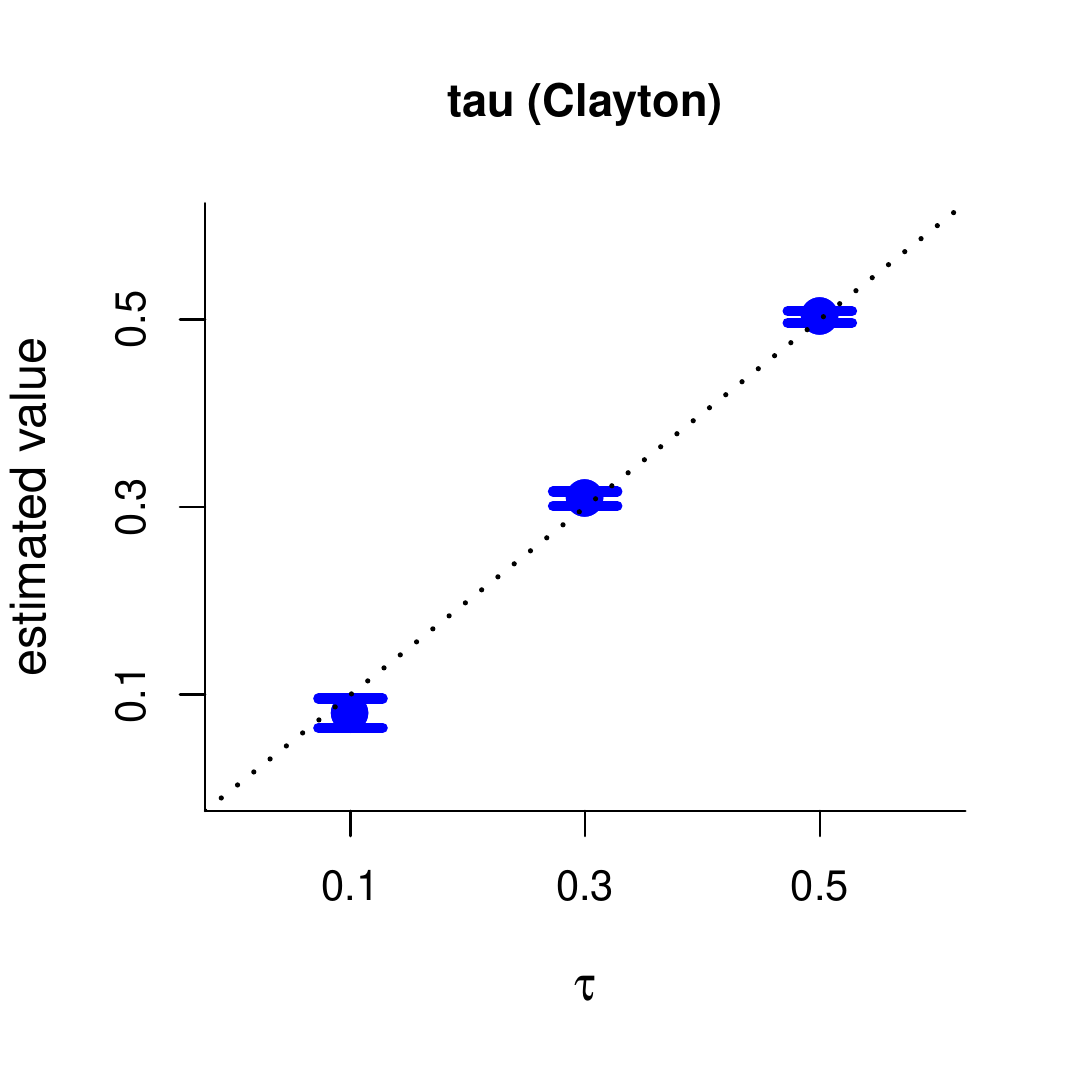}\includegraphics[width=5cm]{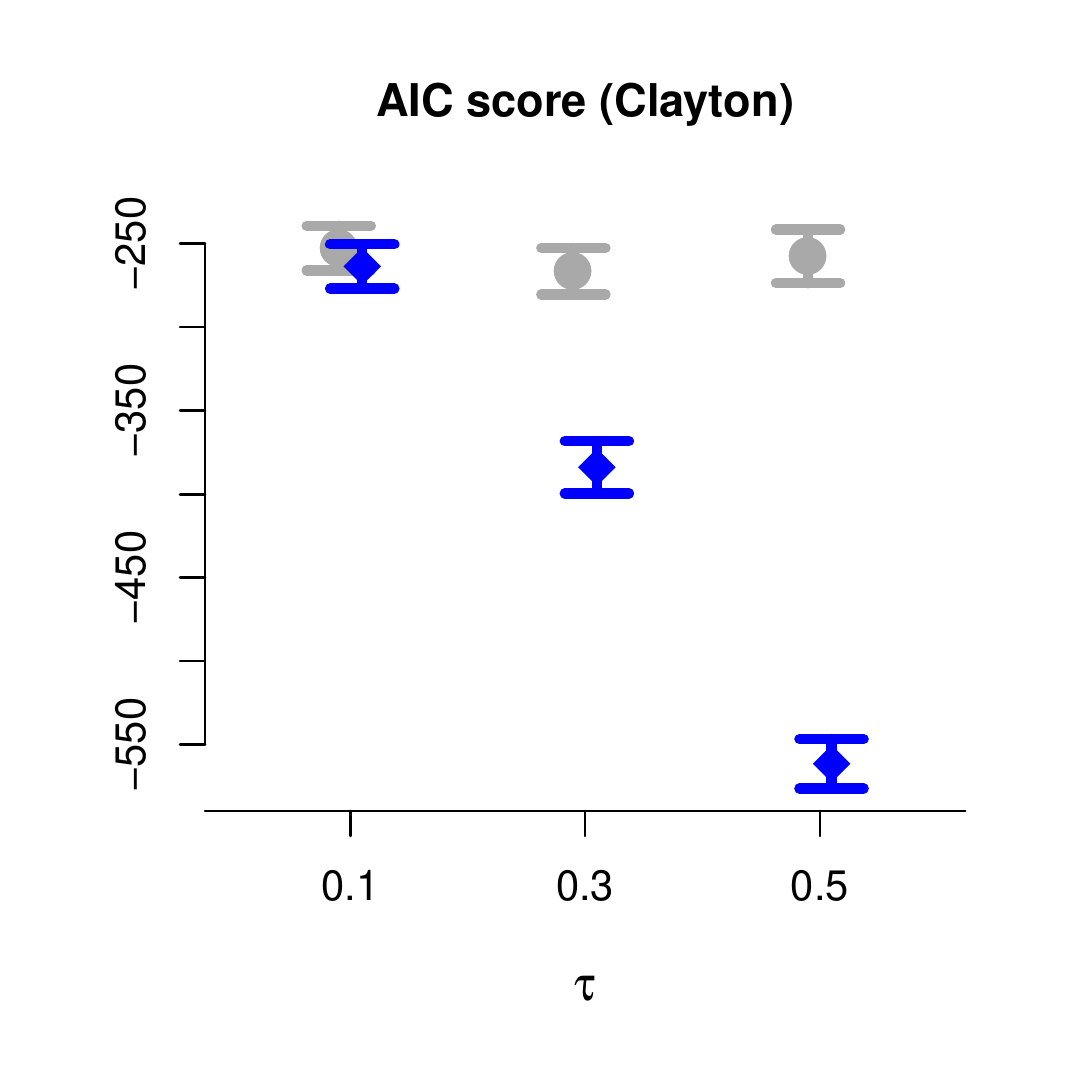}\includegraphics[width=5cm]{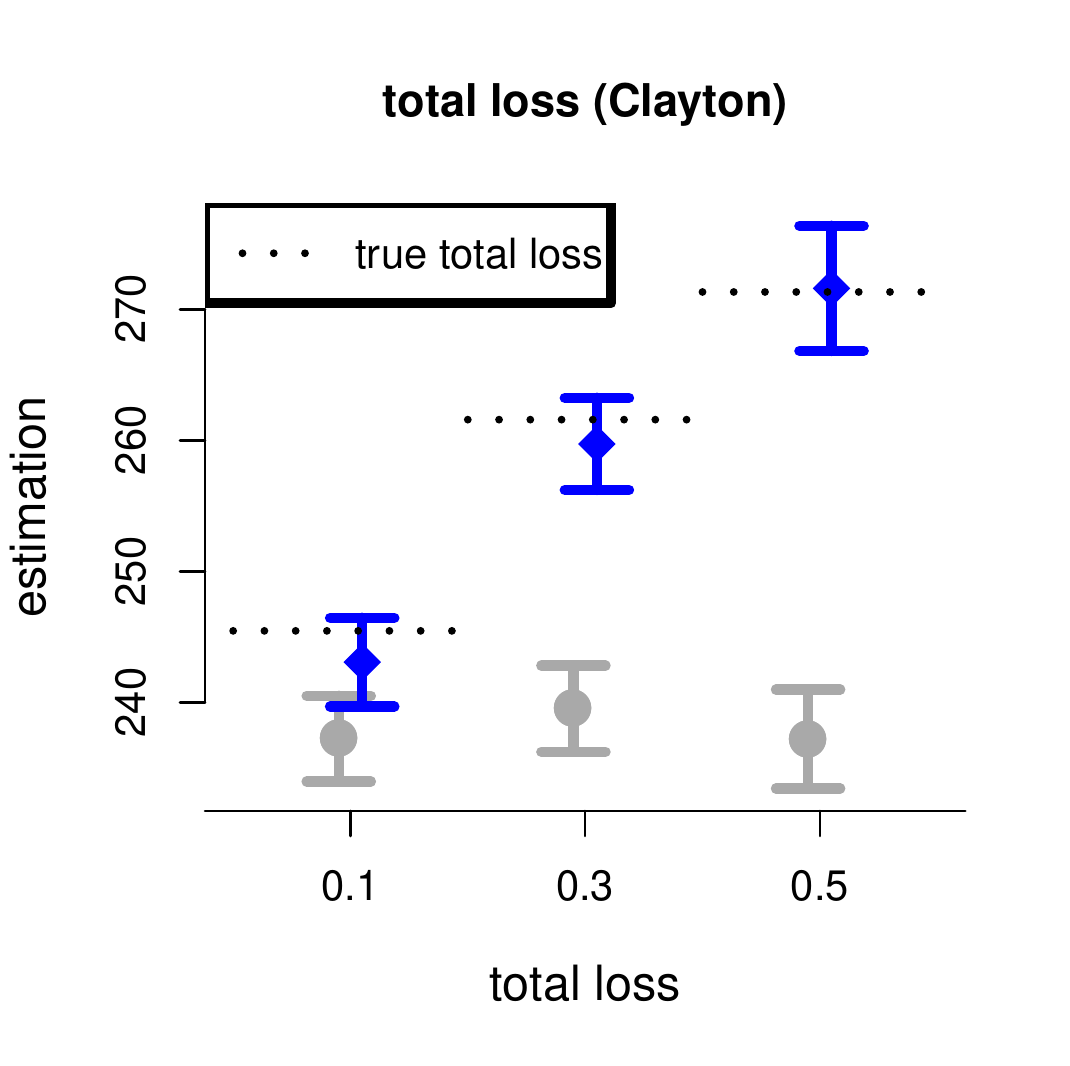}
\end{center}
\caption{Results of the simulation study for the Clayton copula. Top row: relative mean squared error \eqref{eq:relmse} for the average claim size (left), the number of claims (center) and then policy loss (right). Bottom row: estimated Kendall's $\tau$ (left), AIC score (center) and estimated total loss (right). We display the mean over $R$ runs. The width of the whiskers is twice the estimated standard deviation of the mean.  Whiskers that are not displayed are too narrow to be visualized.}
\label{fig:clayton}%
\end{figure}

Figure \ref{fig:clayton} displays the results for the Clayton copula. For each quantity that we compute in each of the $R$ simulation runs, we display the mean over all $R$ runs. The means are indicated by a square. The width of each error bar equals twice the standard deviation of the quantity, divided by $\sqrt{R}$.

The upper row in Figure \ref{fig:clayton} displays the relative mean-squared error of $\widehat {\bm \alpha}$, $\widehat {\bm \beta}$, and the estimated expected policy loss. Overall, we observe that the relative mean squared error for the regression parameters (left and center panel) are not significantly different . For the policy loss however (right panel), the relative mean squared error is  lower for the joint, copula-based model, and this improvement becomes more pronounced for higher values of $\tau$.

The first column of the second row displays the estimated value of Kendall's $\tau$. Here, the dashed line indicates the true value of Kendall's $\tau$. We observe that the  estimation of $\tau$ is very good.  Moreover, the AIC score (center panel in the second row) of the joint model is lower than the one of the marginal models . This shows that if joint model is the true model, then we have to use the joint estimation approach, i.e. the dependence cannot be ignored. The right lower panel displays the estimated total loss. The dashed horizontal lines are the true values of the total loss for the respective value of Kendall's $\tau$. We observe that the independence model systematically underestimates the total loss. This confirms the conclusions drawn  from Figure \ref{fig:expected_policy_loss}.

The results for the three other copula families confirm all the findings made for the Clayton copula. We display the results in  \ref{app:figures}.
\section{Case study: car insurance data}\label{sec:real}
We consider data provided by a German insurance company. It contains car insurance data for $7663$ German insurance policy groups from the year 2000. It contains seven covariates and information on the exposure time. All seven covariates are categorical. The data was previously analyzed by \cite{czakas10}.  Details on the covariates are given in Table \ref{tab:covariates}.
\begin{table}[t]
\begin{center}
\begin{tabular}{ccc}
\hline
name &description &number of categories\\
\hline
\texttt{gen} & driver's gender &2\\
\texttt{rcl} & regional class &8\\
\texttt{bonus}& no-claims bonus &7\\
 \texttt{ded}& type of deductible &5\\
\texttt{dist} &distance driven &5\\
\texttt{age} & driver's age &6\\
\texttt{const} & construction year of the car &7
\end{tabular}
\end{center}
\caption{Covariates in the German car insurance data set.}
\label{tab:covariates}
\end{table}
\subsection{Marginal models}\label{subsec:marmodel}
We first analyze the marginal models. We fit a Gamma regression model for the average claim size, and a zero-truncated Poisson regression model for the number of claims. Next, we investigate the significance of the estimated regression parameters $\widehat{\bm \alpha}$ and $\widehat{\bm \beta}$. We are interested in those coefficients that are significantly different from $0$. Recall (see Section \ref{subsec:ci}) that asymptotically, these estimates are normally distributed, and that we can construct approximate confidence intervals using the observed Fisher information. In addition, we adjust the tests for  multiple comparisons and the dependence of the estimators \citep{Hothorn08}. For the number of claims, the covariates age and construction year  do not have any significant coefficients on a level of $\alpha=0.05$.  Wit re-fit the marginal models, leaving out the respective non-significant covariates. Figure \ref{fig:ci} displays the joint $95 \%$ confidence intervals of the coefficients, showing that the remaining covariates are significant on the $5\%$-level.

\begin{figure}[t]
\begin{center}
\includegraphics[width=7.5cm]{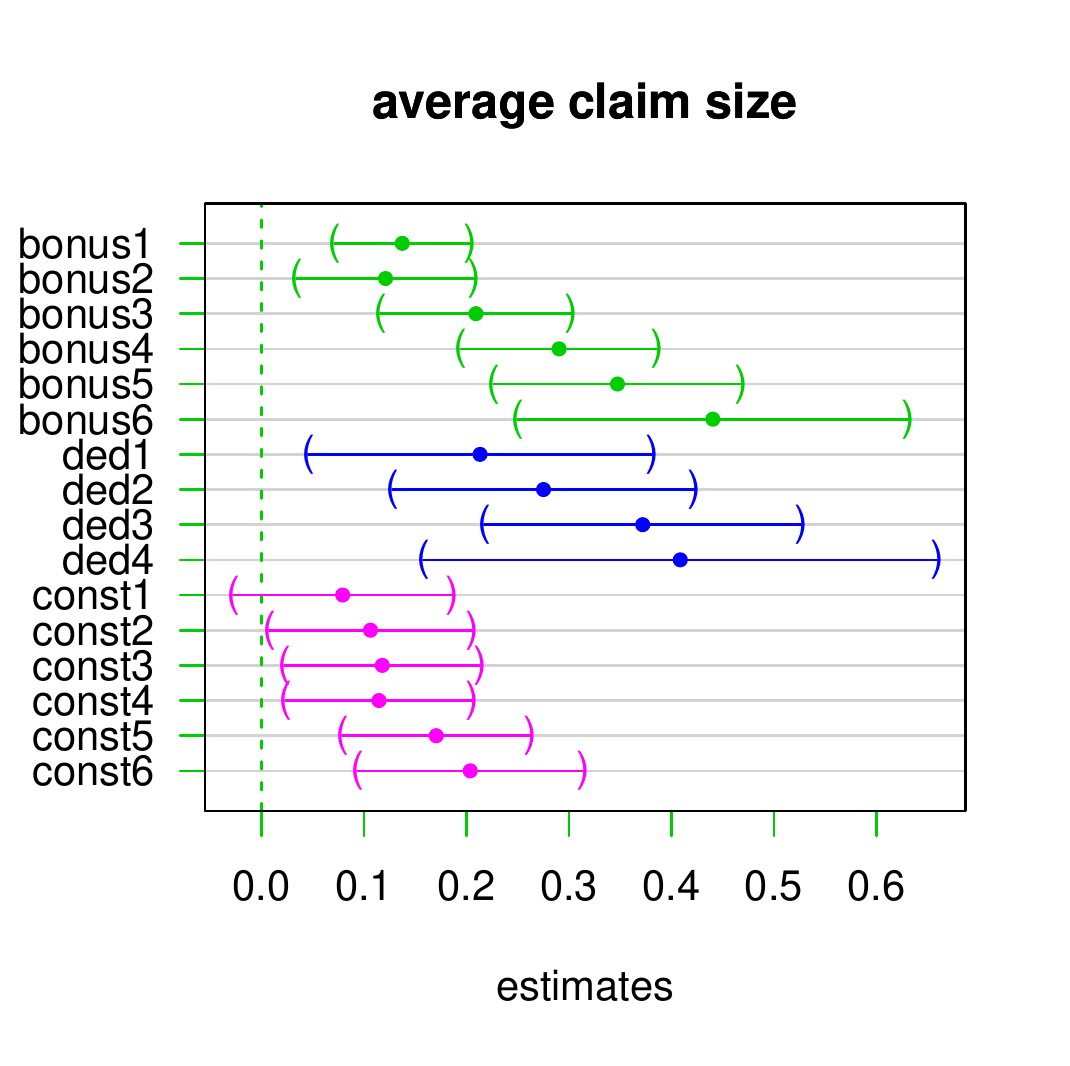}\includegraphics[width=7.5cm]{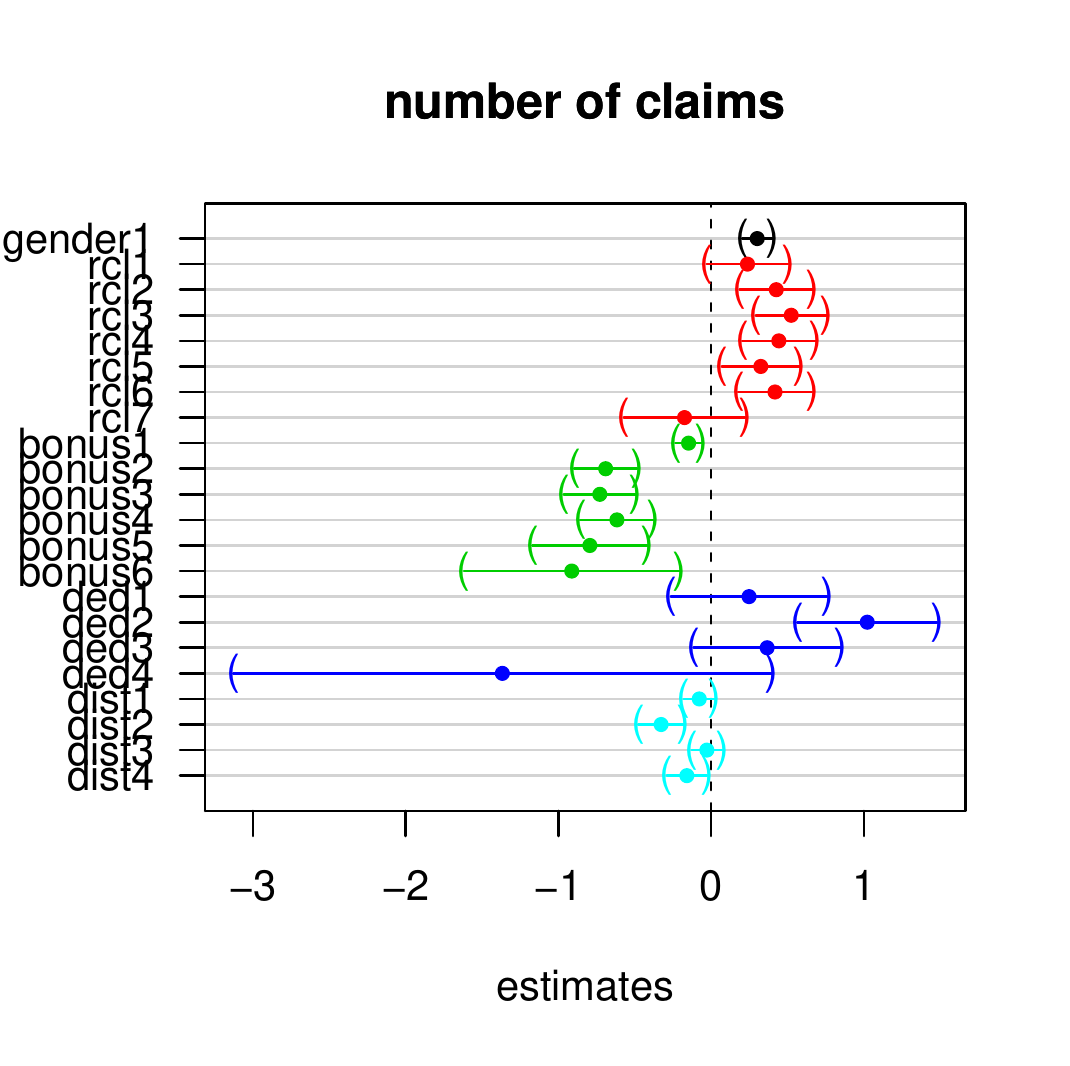}
\end{center}
\caption{Marginal regression models. Joint $95 \%$ confidence intervals for the regression coefficients. Left: Average claim size. Right: number of claims.}
\label{fig:ci}
\end{figure}

\subsection{Joint model}
We use the covariates selected in Section \ref{subsec:marmodel} and fit the joint regression model for each copula family. For each pair of copula families, we perform a corresponding Vuong test. Table \ref{tab:vuong} displays the results. For each pair, we display the copula family that is selected on a $\alpha=0.05$-level. In parentheses, we display the value of the Vuong test statistic \eqref{eq:vuong}. Note that a value $>2$ indicates that we select model 1, and that a value $<-2$ indicates that we select model 2.
\begin{table}[b]
{\scriptsize{\begin{center}
\begin{tabular}{cc||c|c|c|c}
&&\multicolumn{4}{c}{model 2}\\
&&Gauss&Clayton&Gumbel&Frank\\
\hline
\hline
\multirow{4}{*}{\rotatebox{90}{model 1}}&Gauss& -&{\it Clayton} (-10.37) & {\it Gauss} (+6.11)& {\it Frank} (-5.34)\\
&Clayton&{\it Clayton} (+10.37)&-& {\it Clayton} (+9.23)& {\it Clayton} (+9.54)\\
&Gumbel&{\it Gauss} (-6.11)&{\it Clayton} (-9.23)&-& {\it Frank} (-6.54)\\
&Frank&{\it Frank} (+5.34)&{\it Clayton} (-9.54)&{\it Frank} (+6.54)&-
\end{tabular}
\end{center}}}
\caption{Pairwise Vuong tests. We display the copula family that is selected on a $\alpha=0.05$-level. In parentheses, we display the value of the Vuong test statistic \eqref{eq:vuong}. A value $>2$ indicates that we select model 1, and that a value $<-2$ indicates that we select model 2.}
\label{tab:vuong}
\end{table}
We conclude that the Clayton copula is preferred to each of the three other copula families. Therefore, for the remainder of this section, we continue our analysis with the Copula family.  The AIC score for the Clayton model and independence model are
\begin{align*}
\text{AIC}_{\text{clayton}}&= 46\,682.35\\
\text{AIC}_{\text{independence}}&=46\,921.67\,.
\end{align*}
In terms of model comparison, this implies that the copula-based model is more appropriate than the independence model. The estimated value of Kendall's $\tau$ $\pm$ its estimated standard deviation equals
\begin{equation*}
\widehat{\tau}_{\text{clayton}}=0.268 \pm 0.098\,,
\end{equation*}
which corresponds to a moderate, positive dependence between average claim sizes and number of claims. As a comparison, we note that the estimated value of Kendall's $\tau$ for the Gauss copula equals $0.157$, which implies that the selection of the copula family has a considerable effect on the estimation of the dependence parameter. Finally, we investigate the impact of this dependence on the estimation of the total loss. For the copula and independence model respectively, we obtain an estimated total loss $\pm$ its estimated standard deviation of
\begin{align*}
\widehat{E}_{\text{Clayton}}\left(\text{T}\right)&= 81\,751.07 \pm 1239.766\\
\widehat{E}_{\text{independence}}\left(\text{T}\right)&=76\,324.45 \pm 1103.301\,.
\end{align*}
As already illustrated in the simulation study, the negligence of the dependency structure leads to considerably lower estimates of the total loss. In our case, this corresponds to a ratio of
\begin{equation*}
\frac{\widehat{E}_{\text{independence}}\left(\text{T}\right)}{\widehat{E}_{\text{Clayton}}\left(\text{T}\right)}=0.934\,,
\end{equation*}
which indicates a possibly severe underestimation of the independence model in presence of frequency-severity dependence.
The more conservative estimate by the copula-based model takes this dependence into account and will thus result in a more appropriate premium rating protecting the insurance company from huge losses in the portfolio.

\section{Summary and discussion}
In this paper, we model average claim sizes and number of claims if both quantities are dependent. We provide exact distributions of individual policy losses, which tend to be left-skew, and -- depending on the parameters of the model -- can be multi-modal. Further, we propose a regression approach that models average claim sizes and number of claims in terms of a set of covariates. We showed theoretically and empirically that the explicit incorporation of the  dependency in terms of copulae has a profound impact on the estimation of the individual policy loss and the total loss.

Our model explicitly incorporates the discrete structure of the number of claims and allows a flexible class of copula families. This extends  previous work that only consider the Gauss copula \citep{czakas10,DeLeonWu2011}. In our case study, we demonstrated that other copula families are more appropriate.

 We stress that our approach does not depend on the particular choice of the marginal distribution or copula family, and it can be extended to other parametric distributions and families (see e.g. \cite{Yee96} and \cite{Yee10} for an overview on appropriate marginal distributions). Moreover, in the case of higher-dimensional mixtures of discrete and continuous random variables, pair-copula constructions \citep{aas09} can be used (see e.g. \cite{pag12} and \cite{jakob12}).

In our  simulation study, we showed that a model that assumes independence of average claim sizes and number of claims  consistently underestimates the total loss of the insurance portfolio implying a severe mispricing of policies. Knowing the true distribution of the policy loss and total loss allows us to correctly assess some risk. This is underpinned in our case study on German car insurance policies. Here, we select relevant covariates for the marginal models and choose the appropriate copula family for the dependence structure using a Vuong test. The data shows a moderate positive dependence. We illustrate that this leads to a more conservative estimation of the total loss, which avoids huge losses in the insurance portfolio and thus possibly filing for bankruptcy. Respecting actuarial prudence therefore requires to take into account possible dependencies between average claim sizes and numbers of claims.

\bibliographystyle{model2-names}
\bibliography{total_loss_submission}

\appendix
\clearpage
%-------------------------------------%
\section{Copulae} \label{app:copulae} %
%-------------------------------------%
Table \ref{tab:overview_copula} provides the definition of the four bivariate copula families, their relationship to Kendall's $\tau$ and their first partial derivative. Here $\Phi_2$ is the cumulative distribution function of the bivariate standard normal distribution, and $\Phi$ is the cumulative distribution function of the univariate standard normal distribution. Further,
\begin{eqnarray*}
D_k(x)&=&\frac{k}{x^k}\int\limits_{0}^{x}\frac{t^k}{e^t-1}dt\,.
\end{eqnarray*}
denotes the Debye function which is defined for $k\in\mathbb{N}$.
\begin{table}[hb]
{\scriptsize{\begin{center}
		\begin{tabular}{lccc}
		\hline
		family & copula $C(u,v,\theta)$ & range of $\theta$&relationship to Kendall's $\tau$\\
		\hline
		Gauss & $\Phi_2\left(\Phi^{-1}(u),\Phi^{-1}(v)|\theta\right)$&$]-1,1[$&$\tau=\frac{2}{\pi}\arcsin(\theta)\in \mathbb{R}$\\
		Clayton & $\left(u^{-\theta}+v^{-\theta}-1\right)^{-1/\theta}$&$]0,\infty[$&  $\tau=\frac{\theta}{\theta+2}\in ]0,\infty[$\\
		Gumbel & $\exp\left(-\left(\left(-\log u\right)^\theta+\left(-\log v\right)^\theta\right)^{1/\theta}\right)$&$[1,\infty[$&$\tau=\frac{\theta-1}{\theta}\in [0,\infty[$\\
		Frank & $-\frac{1}{\theta}\log\left(1+\frac{\left(e^{-\theta u}-1\right)\left(e^{-\theta v}-1\right)}{e^{-\theta}-1}\right)$&$\mathbb{R}\backslash\{0\}$&$\tau=1-\frac{4}{\theta}\left[1-D_1(\theta)\right]\in \mathbb{R}\setminus\{0\}$
		\end{tabular}
\end{center}}}
		\caption{Characteristics of selected copula families. }
		\label{tab:overview_copula}
\end{table}
\begin{table}[hb]
\begin{center}
		\begin{tabular}{lc}
		\hline
		family &first partial derivative $D_1(u,v|\theta)$ \\
		\hline
		Gauss &  $\Phi\left(\frac{\Phi^{-1}(v)-\theta\Phi^{-1}(u)}{\sqrt(1-\theta^2)}\right)$\\
		Clayton & $\left(u^{-\theta}+v^{-\theta}-1\right)^{-1/\theta-1}u^{-\theta-1}$\\
		Gumbel & $u^{-1}\exp\left(-\left(\left(-\log u\right)^\theta+\left(-\log v\right)^\theta\right)^{1/\theta}\right)$\\
		Frank &  $\frac{e^{\theta}\left(e^{\theta v}-1\right)}{e^{\theta(u+1)}+e^{\theta(v+1)}-e^\theta-e^{\theta(u+v)}}$
		\end{tabular}
		\caption{First partial derivative  of selected copula families.}
		\label{tab:der_copula}
\end{center}
\end{table}
\clearpage
\section{Results of the simulation study}\label{app:figures}
We display the  results for the Gauss copula (Figure \ref{fig:gauss}), the Gumbel copula (Figure \ref{fig:gumbel}) and the Frank copula (Figure \ref{fig:frank}).
\begin{figure}[hb]
\begin{center}
\includegraphics[width=5cm]{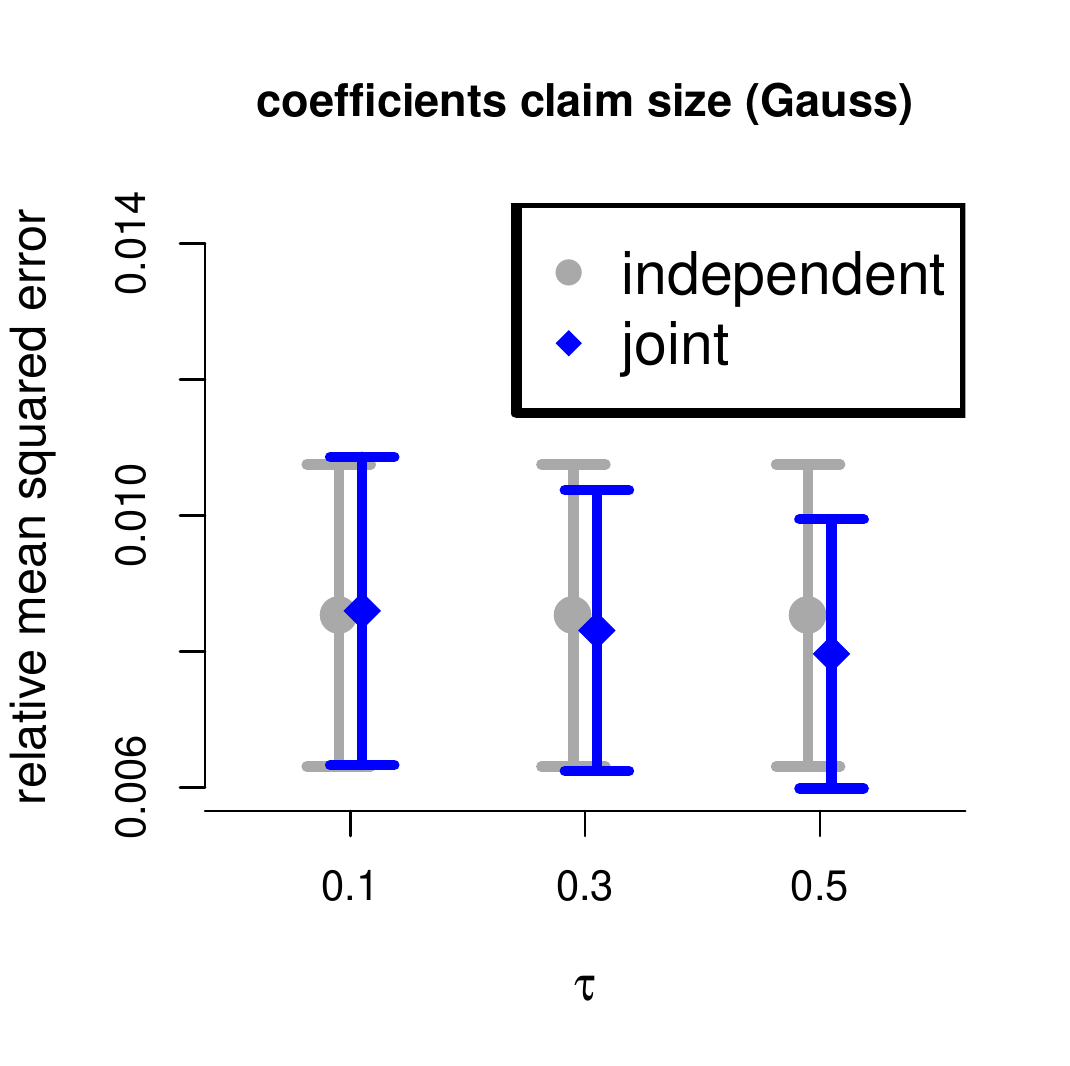}\includegraphics[width=5cm]{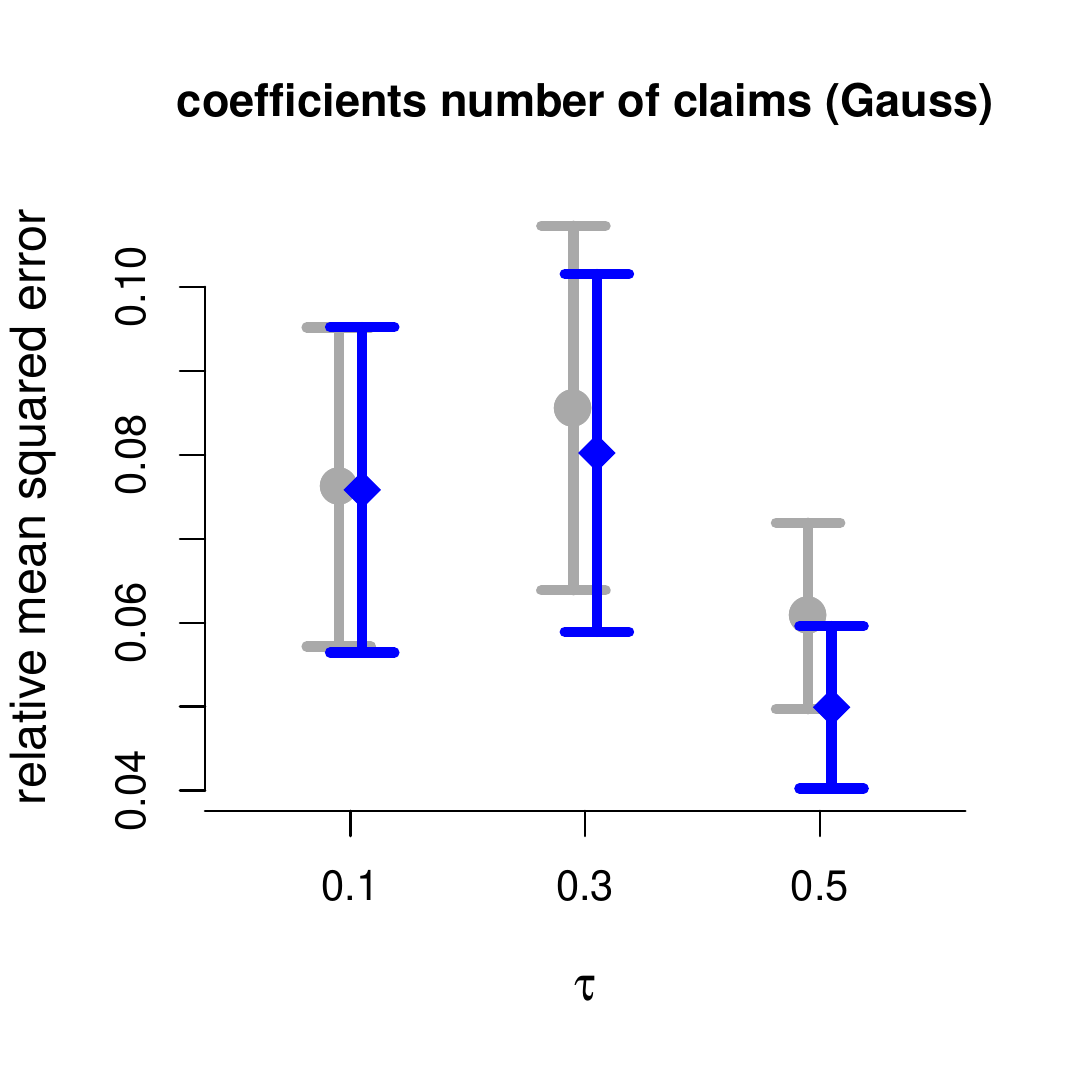}\includegraphics[width=5cm]{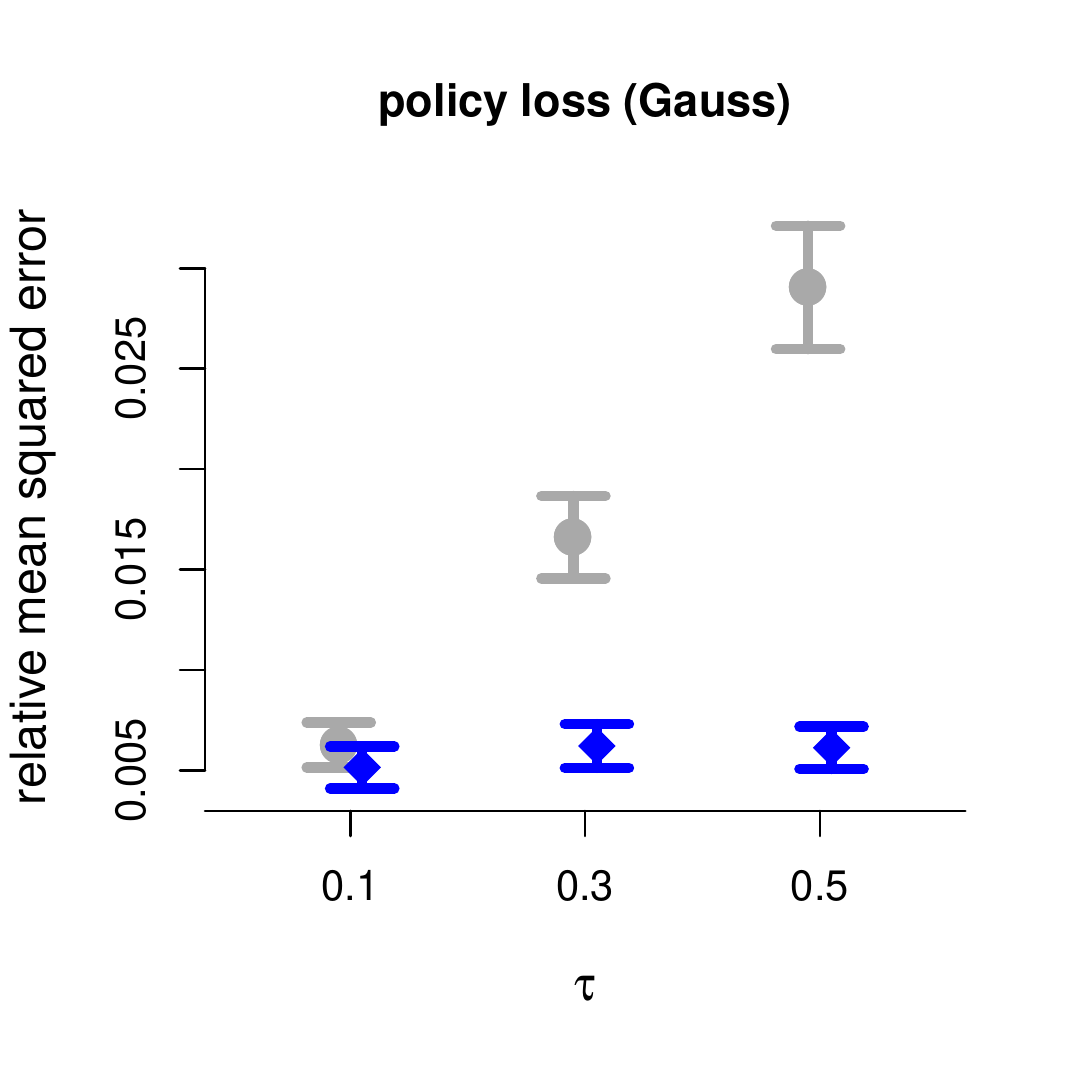}\\
\includegraphics[width=5cm]{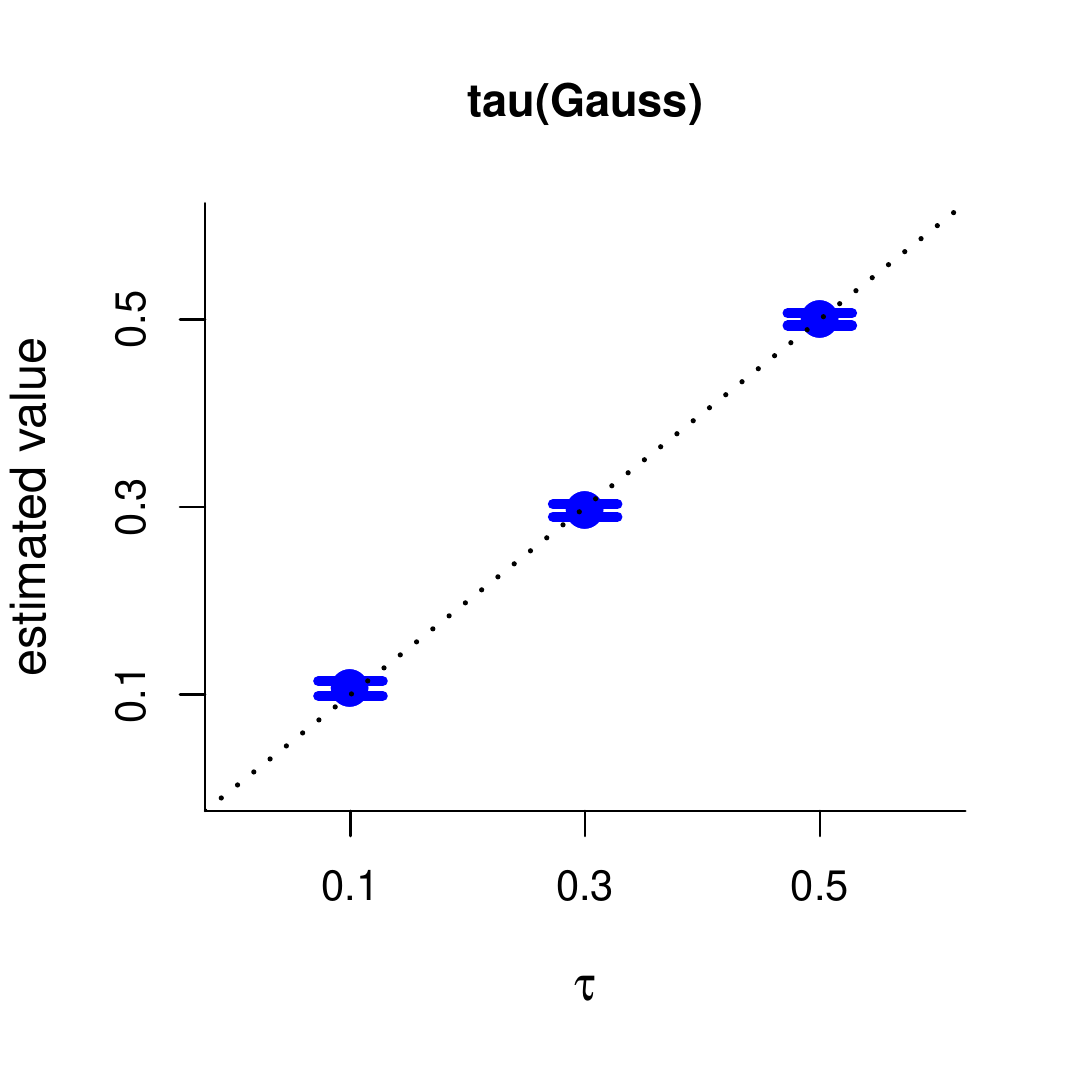}\includegraphics[width=5cm]{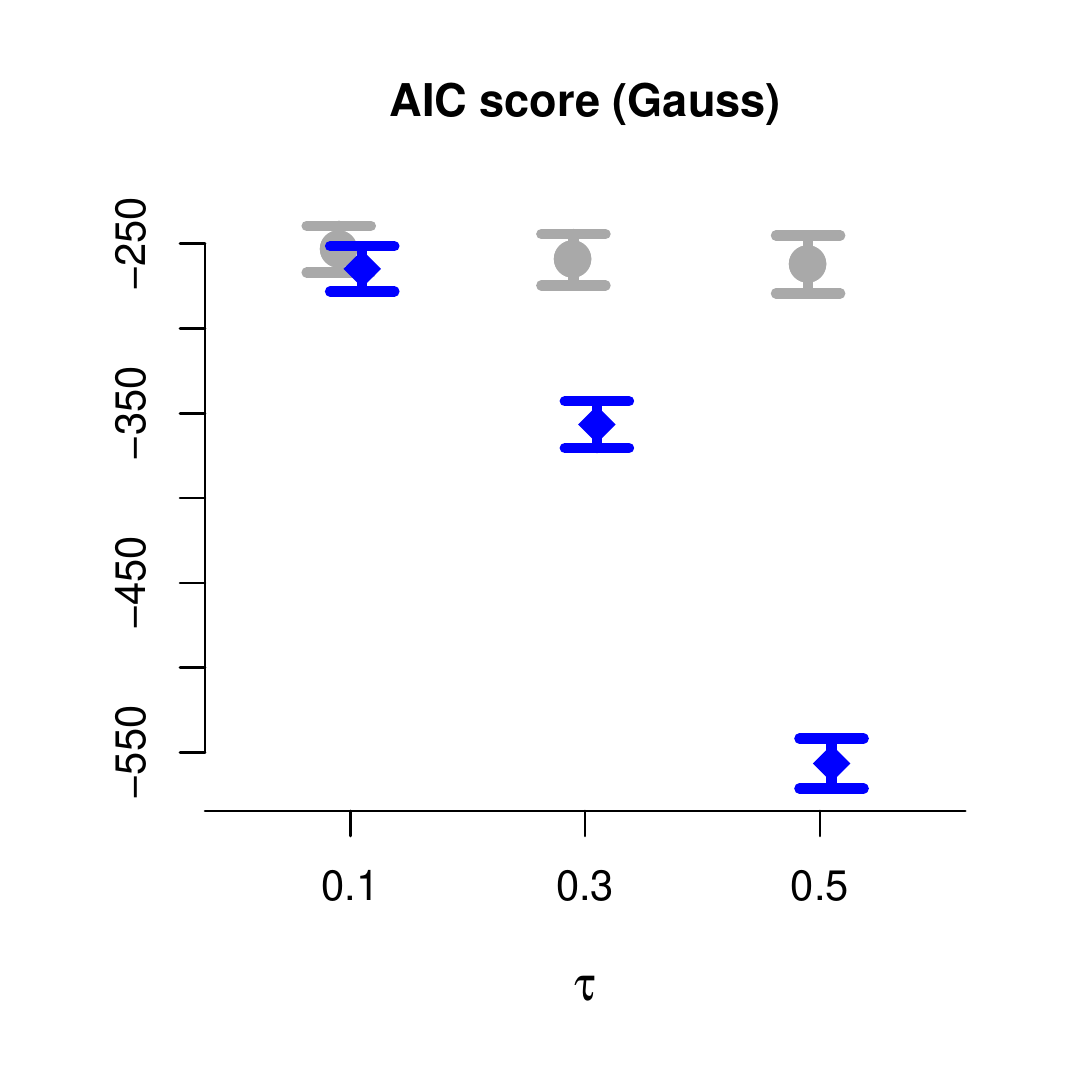}\includegraphics[width=5cm]{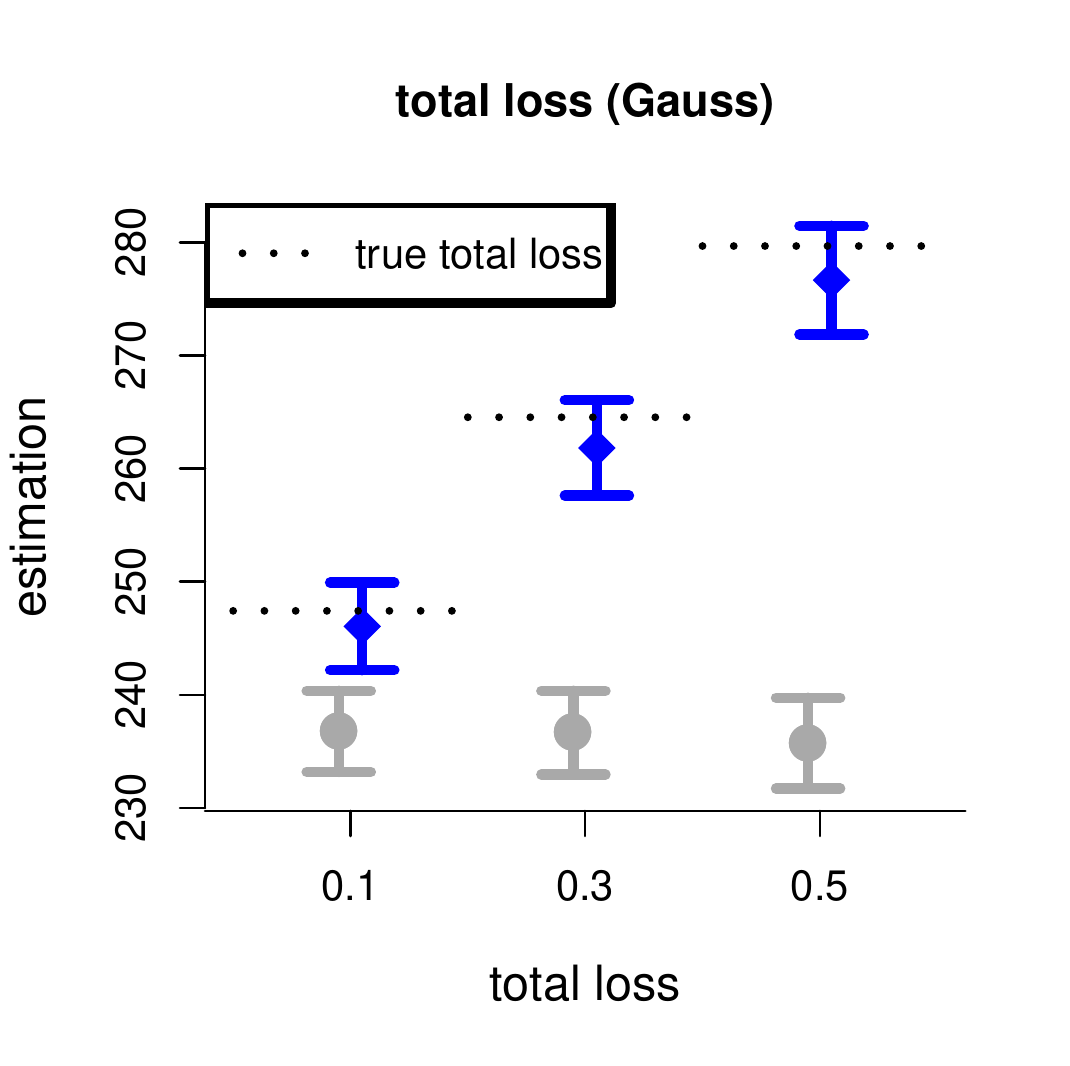}
\end{center}
\caption{Results of the simulation study for the Gauss copula. Top row: relative mean squared error \eqref{eq:relmse} for the average claim size (left), the number of claims (center) and then policy loss (right). Bottom row: estimated Kendall's $\tau$ (left), AIC score (center) and estimated total loss (right). We display the mean over $R$ runs. The width of the whiskers is twice the estimated standard deviation of the mean.  Whiskers that are not displayed are too narrow to be visualized.}
\label{fig:gauss}%
\end{figure}
\begin{figure}[t]
\begin{center}
\includegraphics[width=5cm]{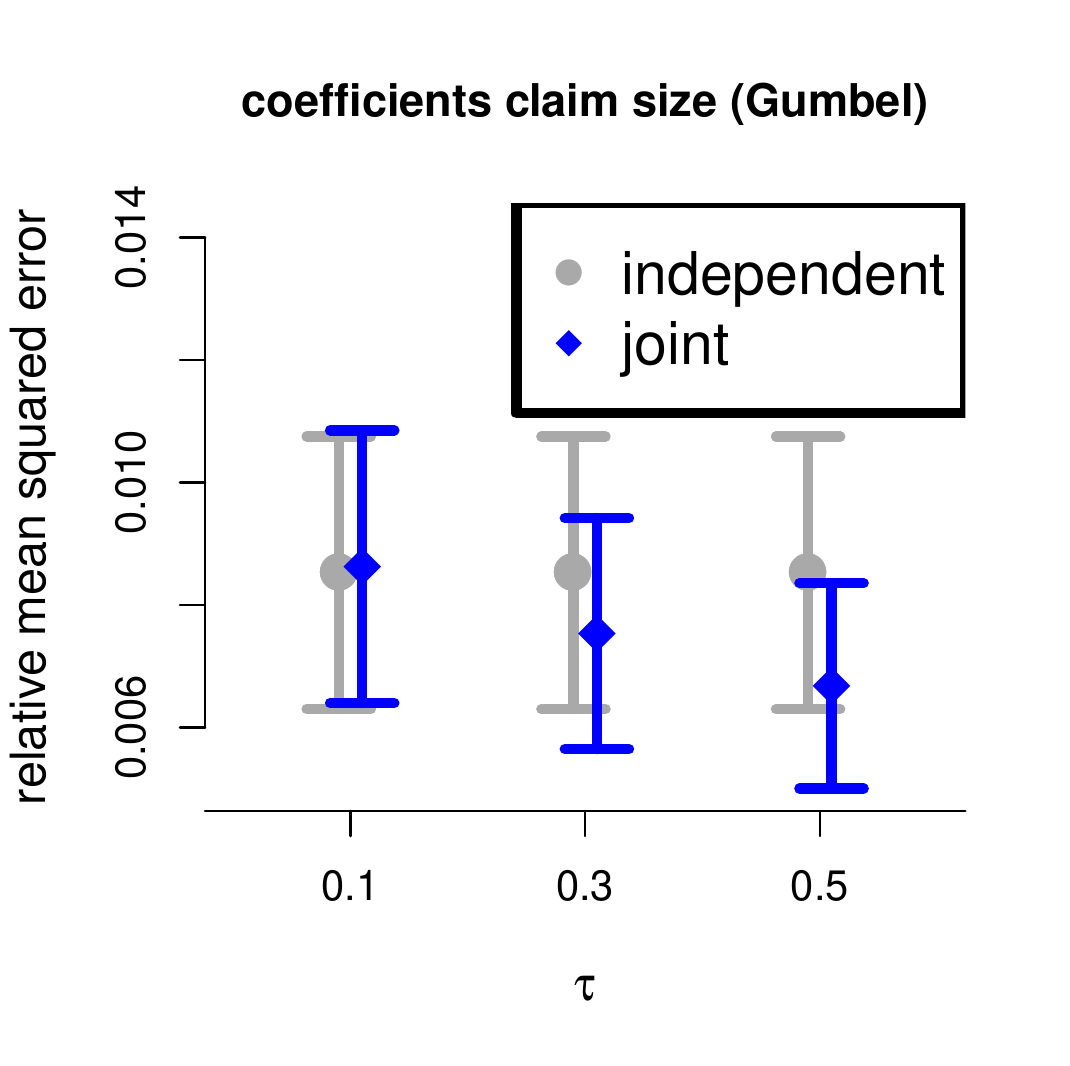}\includegraphics[width=5cm]{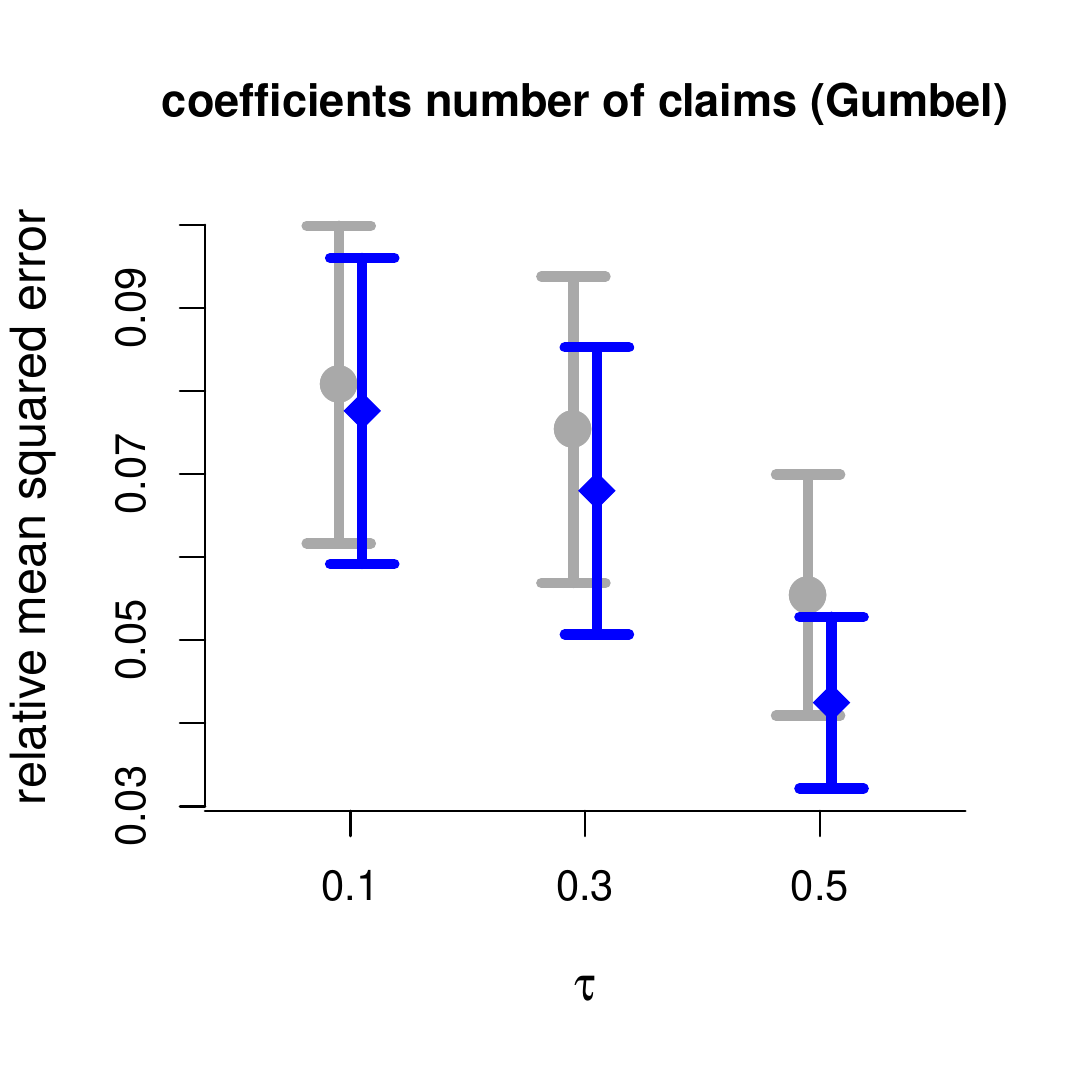}\includegraphics[width=5cm]{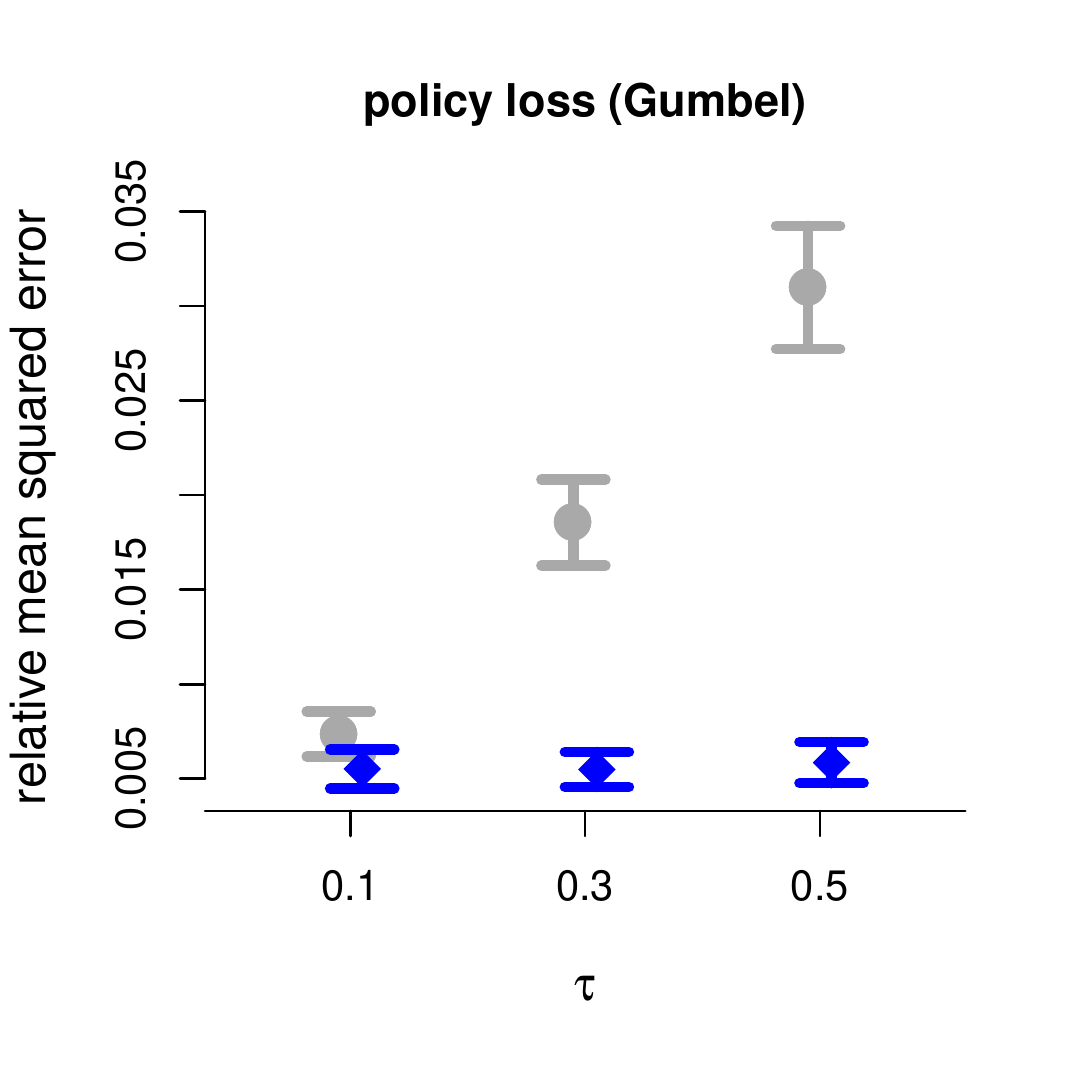}\\
\includegraphics[width=5cm]{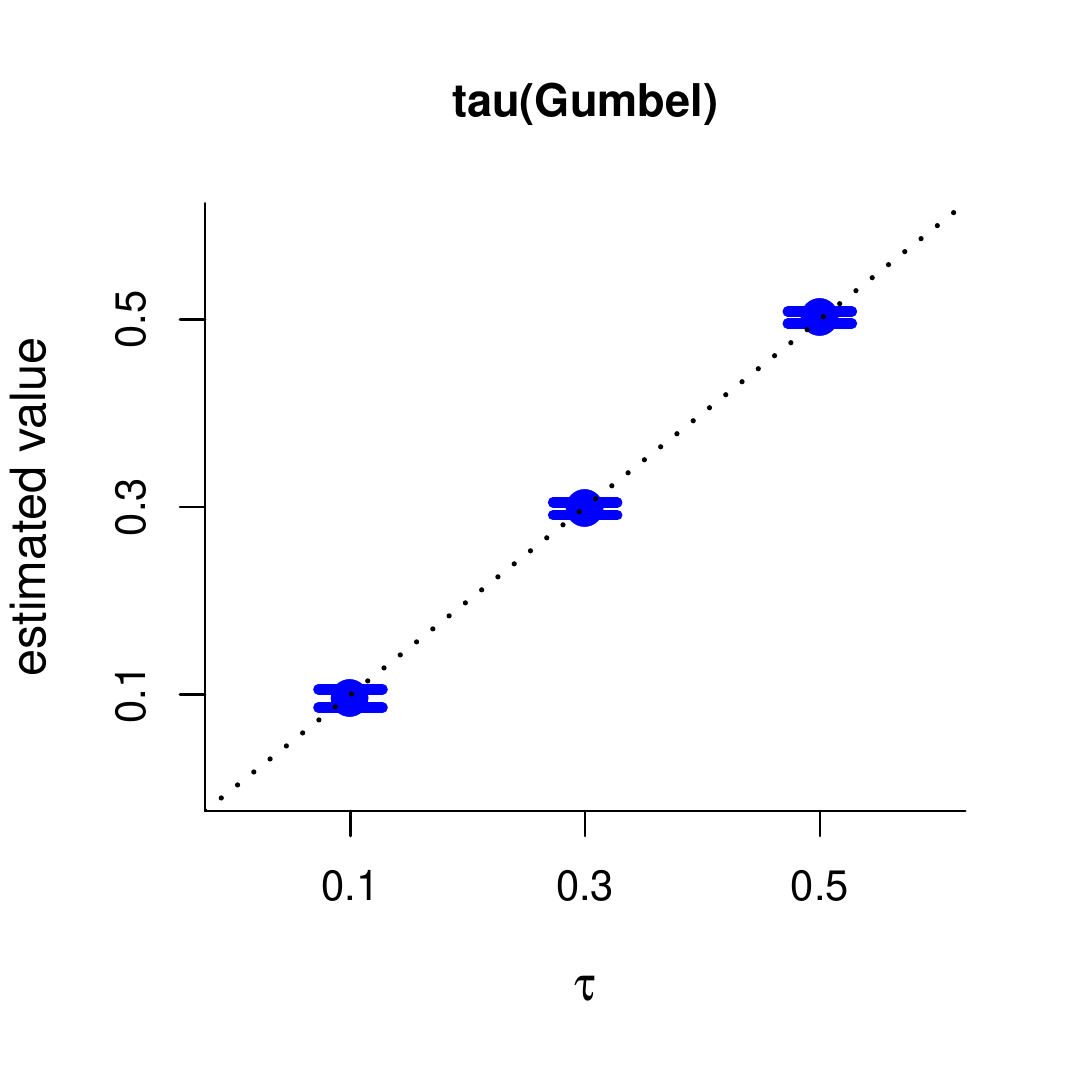}\includegraphics[width=5cm]{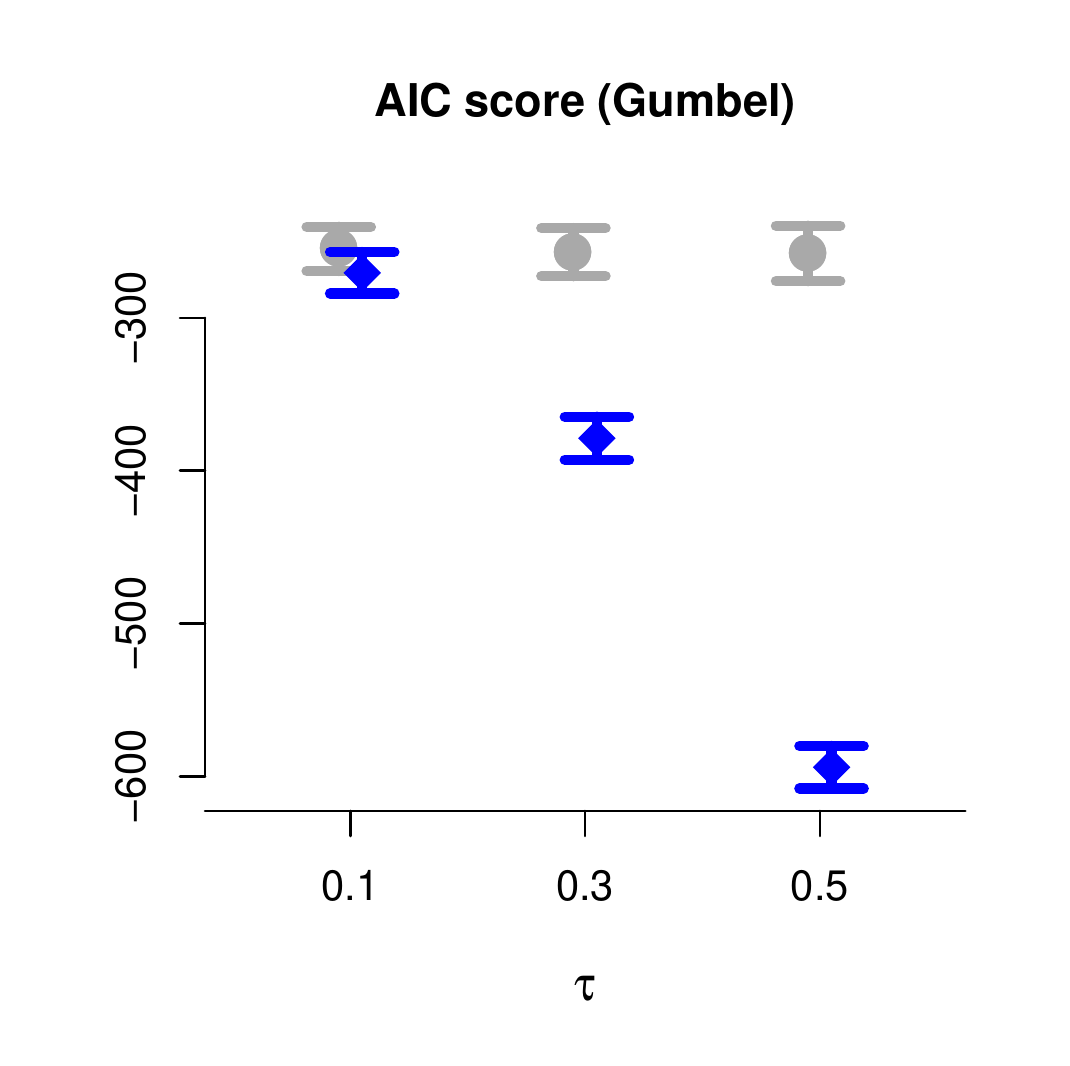}\includegraphics[width=5cm]{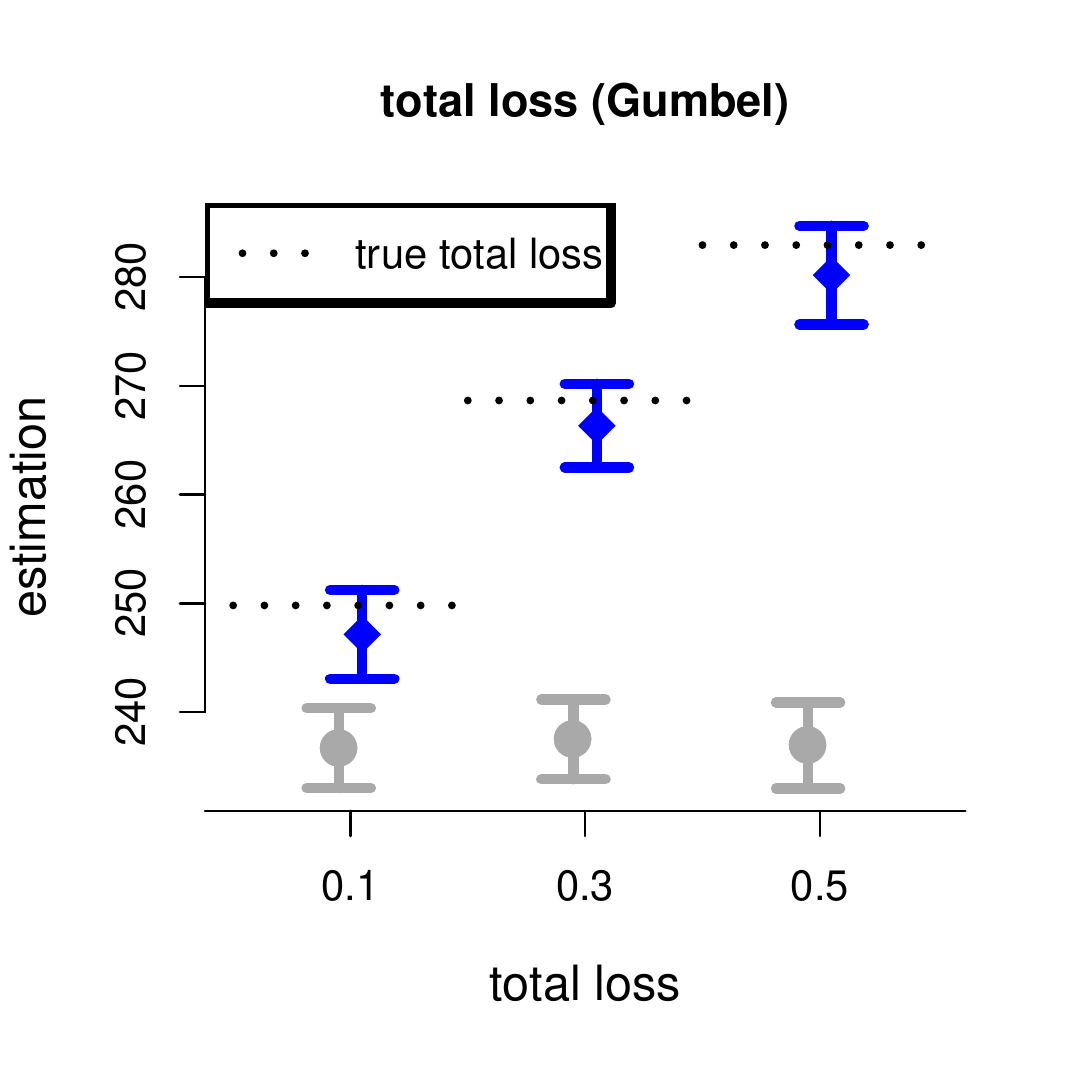}

\end{center}
\caption{Results of the simulation study for the Gumbel copula. Top row: relative mean squared error \eqref{eq:relmse} for the average claim size (left), the number of claims (center) and then policy loss (right). Bottom row: estimated Kendall's $\tau$ (left), AIC score (center) and estimated total loss (right). We display the mean over $R$ runs. The width of the whiskers is twice the estimated standard deviation of the mean.  Whiskers that are not displayed are too narrow to be visualized.}
\label{fig:gumbel}%
\end{figure}

\begin{figure}%
\begin{center}
\includegraphics[width=5cm]{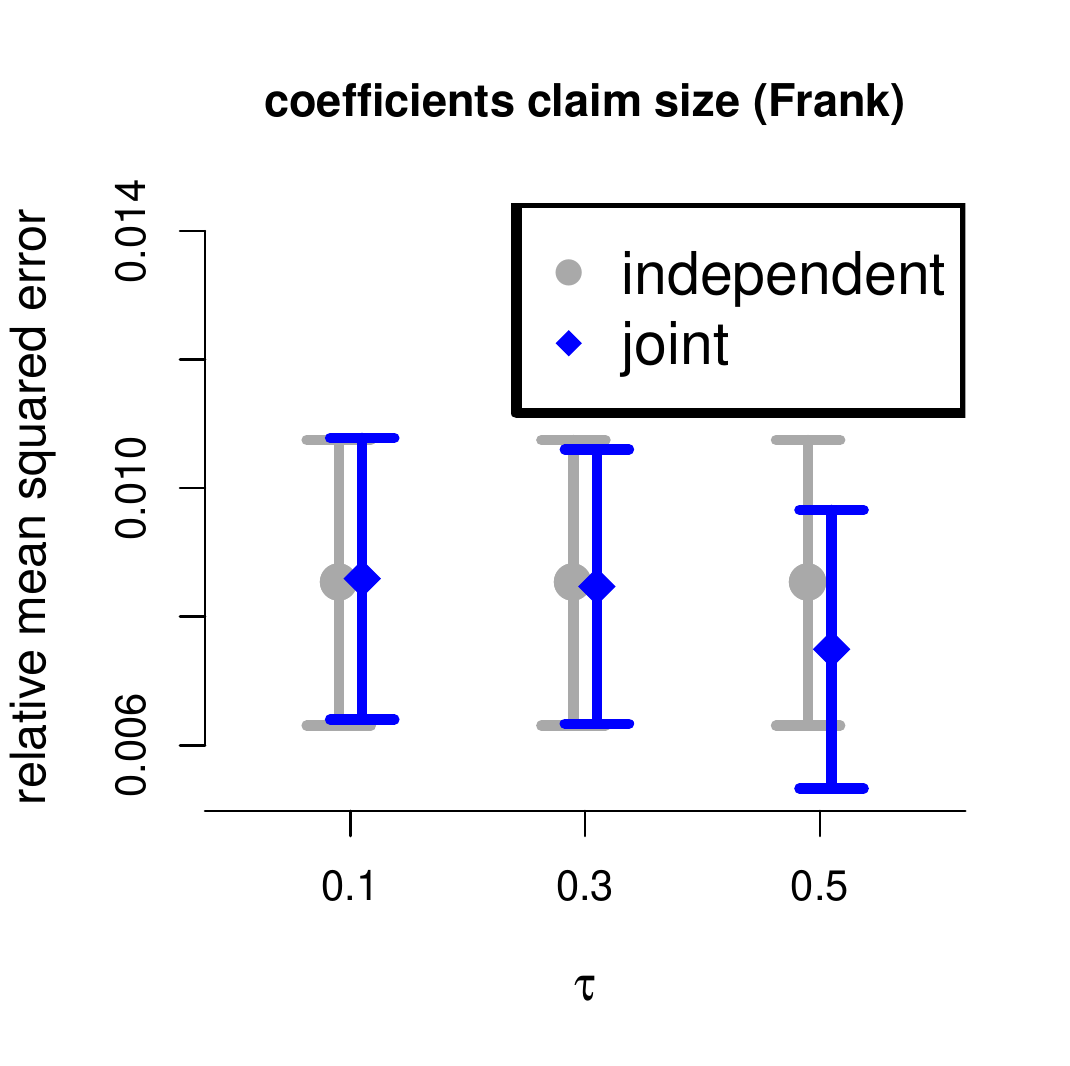}\includegraphics[width=5cm]{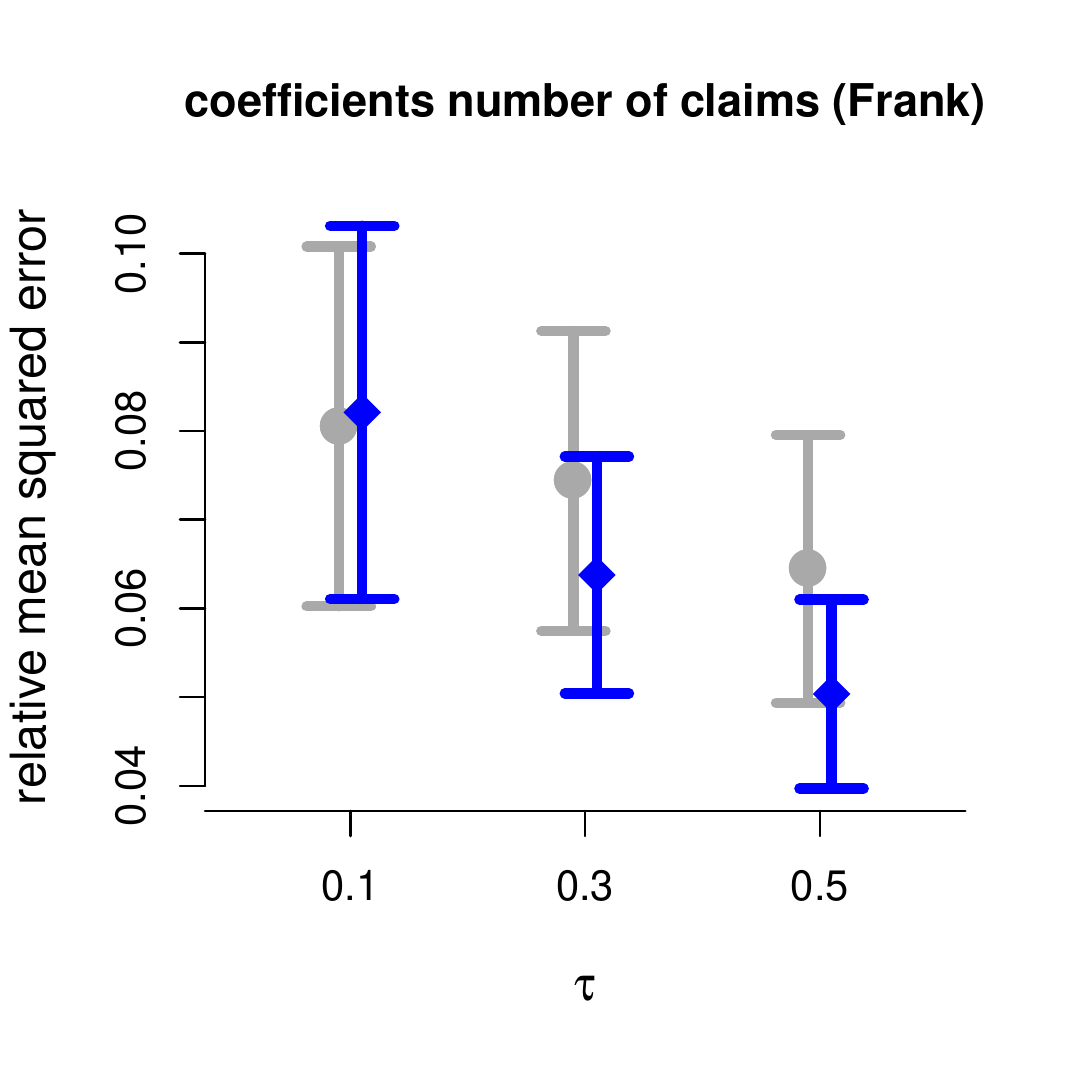}\includegraphics[width=5cm]{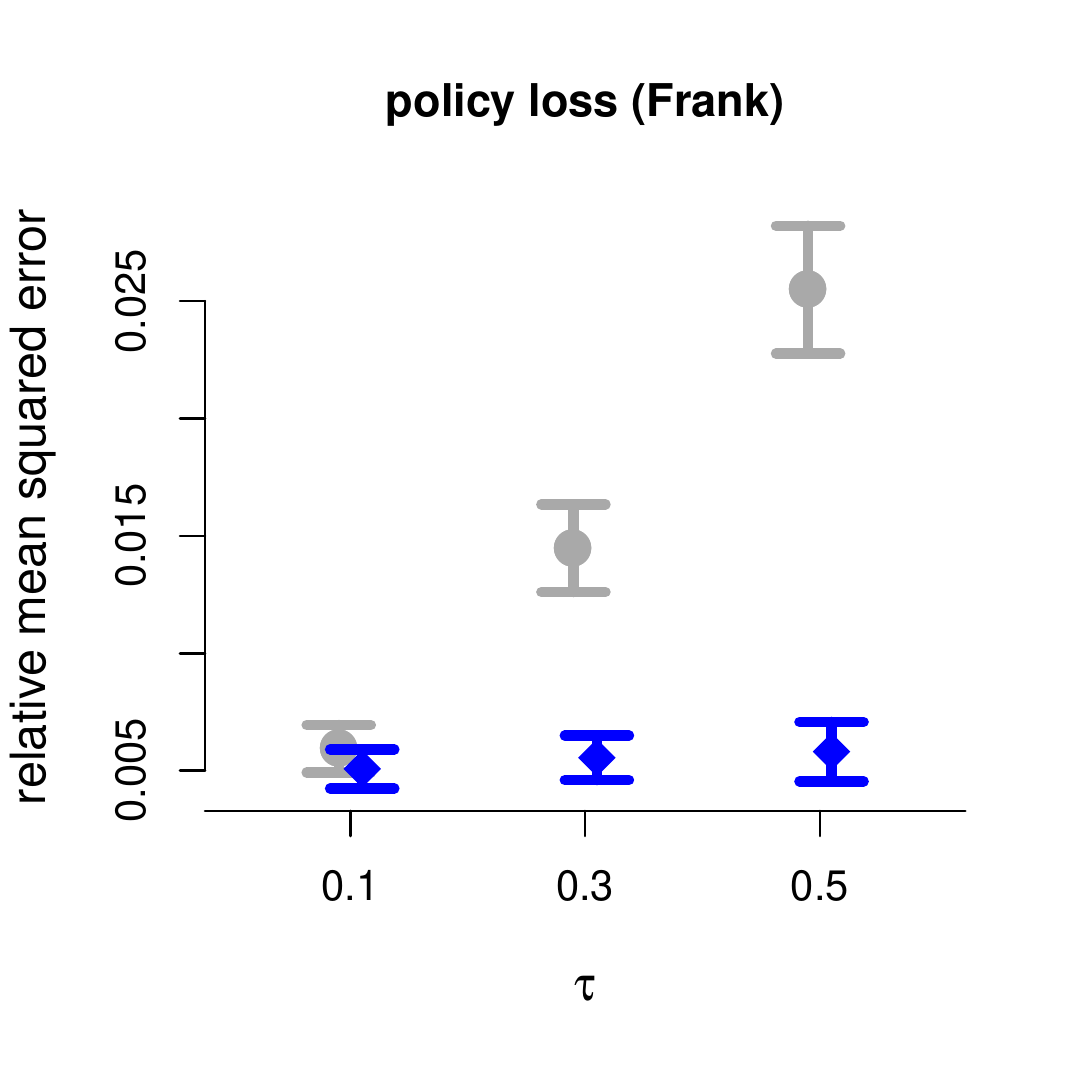}\\
\includegraphics[width=5cm]{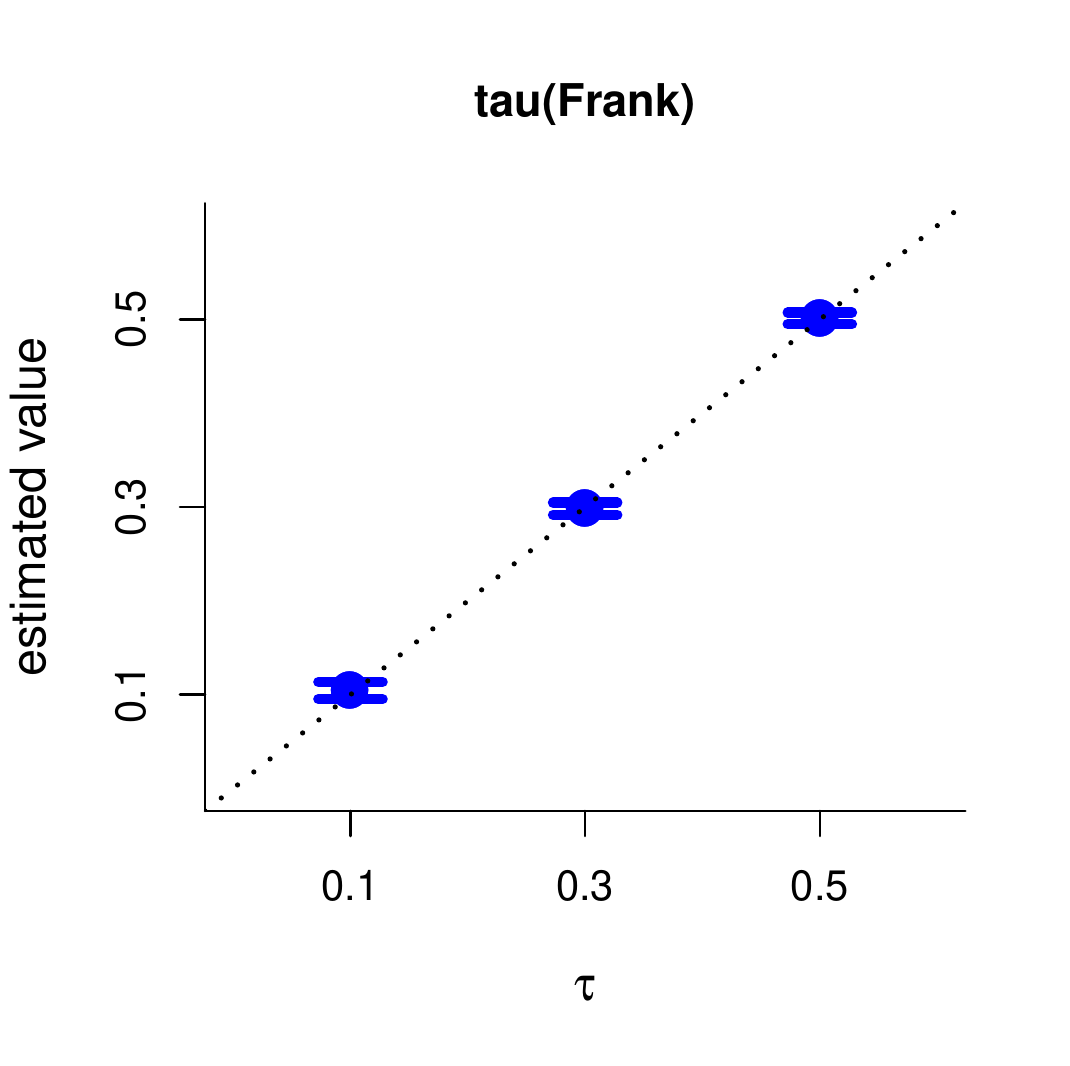}\includegraphics[width=5cm]{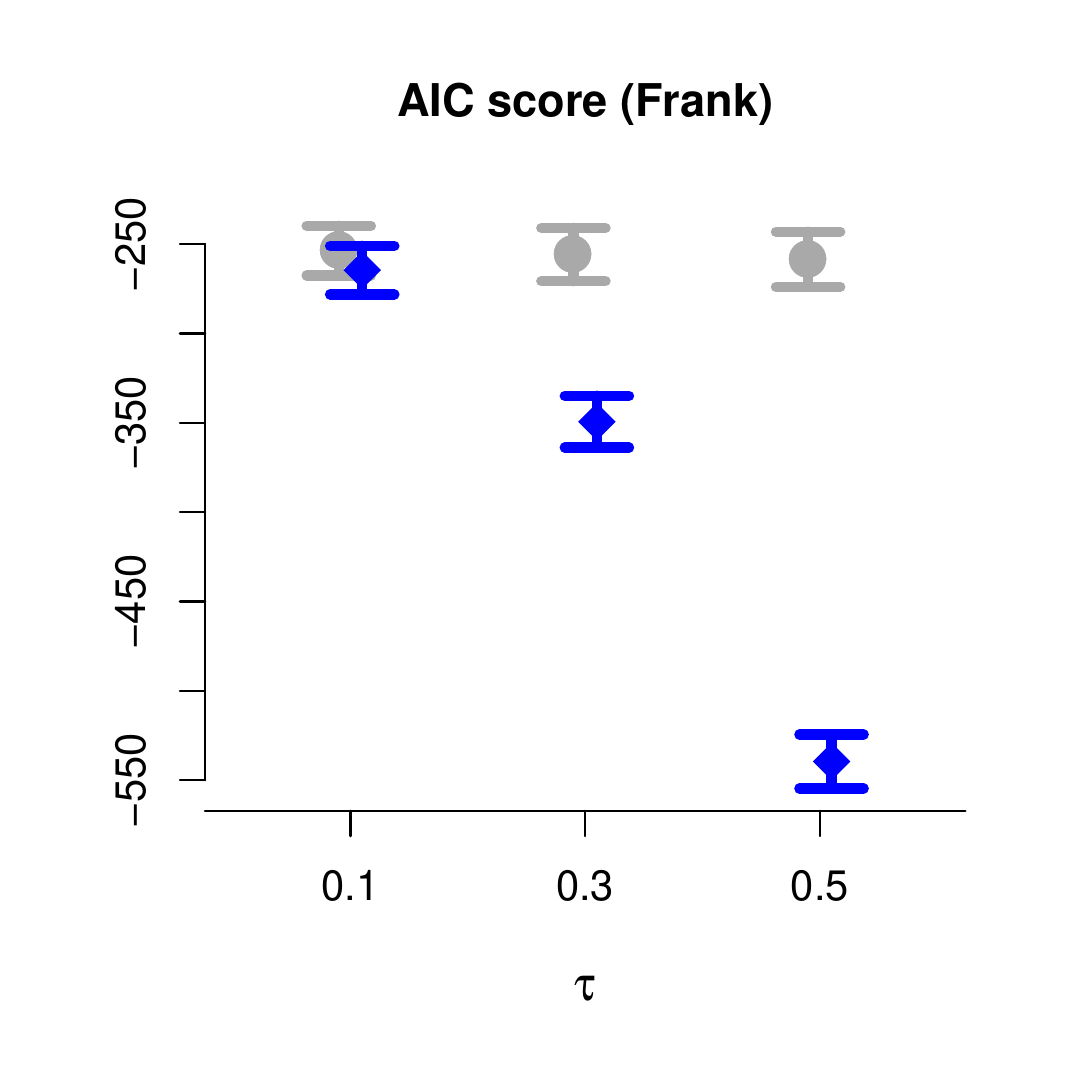}\includegraphics[width=5cm]{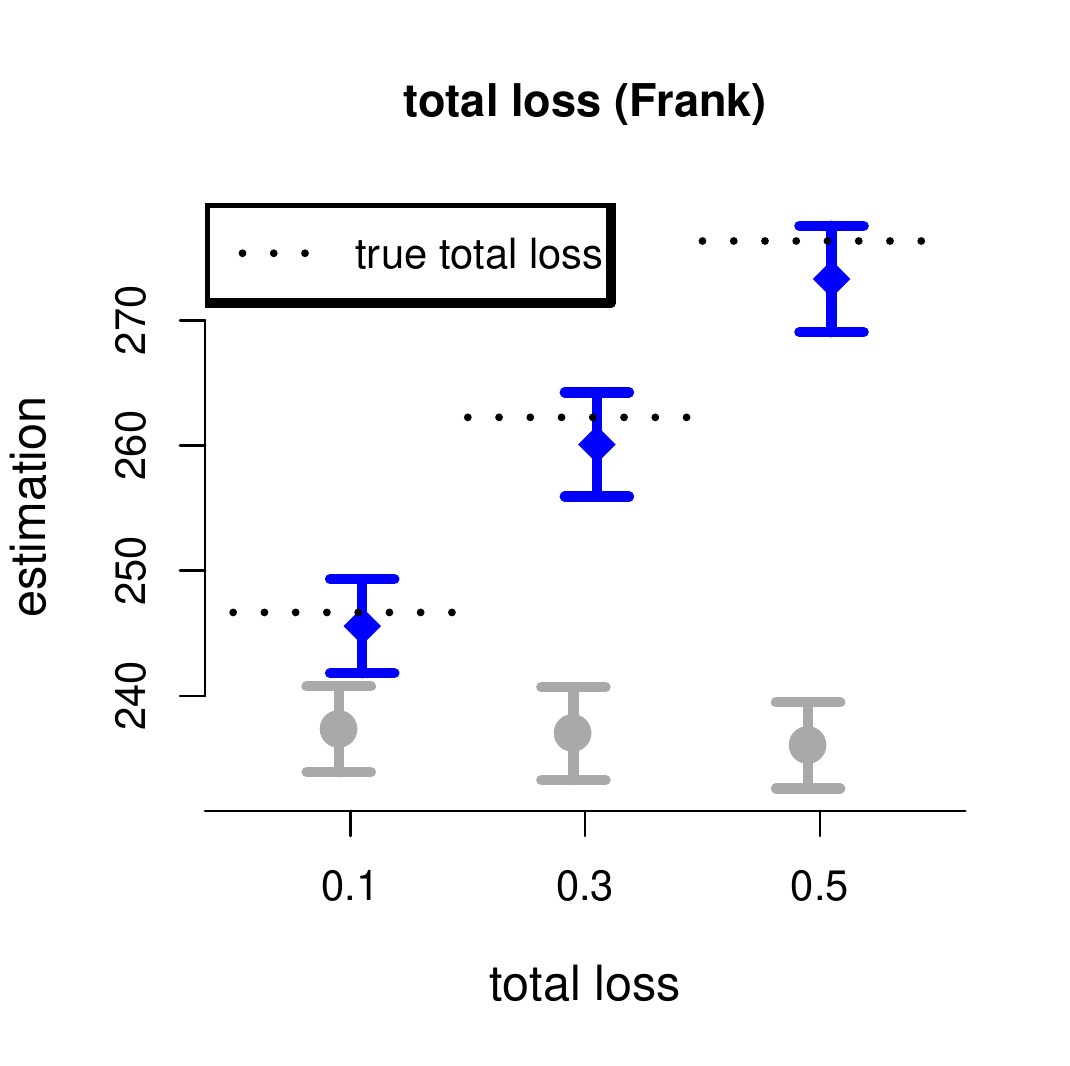}
\end{center}
\caption{Results of the simulation study for the Frank copula.Top row: relative mean squared error \eqref{eq:relmse} for the average claim size (left), the number of claims (center) and then policy loss (right). Bottom row: estimated Kendall's $\tau$ (left), AIC score (center) and estimated total loss (right). We display the mean over $R$ runs. The width of the whiskers is twice the estimated standard deviation of the mean.  Whiskers that are not displayed are too narrow to be visualized.}
\label{fig:frank}%
\end{figure}

%% Authors are advised to submit their bibtex database files. They are
%% requested to list a bibtex style file in the manuscript if they do
%% not want to use model2-names.bst.

%% References without bibTeX database:

% \begin{thebibliography}{00}

%% \bibitem must have one of the following forms:
%%   \bibitem[Jones et al.(1990)]{key}...
%%   \bibitem[Jones et al.(1990)Jones, Baker, and Williams]{key}...
%%   \bibitem[Jones et al., 1990]{key}...
%%   \bibitem[\protect\citeauthoryear{Jones, Baker, and Williams}{Jones
%%       et al.}{1990}]{key}...
%%   \bibitem[\protect\citeauthoryear{Jones et al.}{1990}]{key}...
%%   \bibitem[\protect\astroncite{Jones et al.}{1990}]{key}...
%%   \bibitem[\protect\citename{Jones et al., }1990]{key}...
%%   \harvarditem[Jones et al.]{Jones, Baker, and Williams}{1990}{key}...
%%

% \bibitem[ ()]{}

% \end{thebibliography}

\end{document}